\documentclass[12pt,a4paper]{article}

\usepackage[verbose]{geometry}
\usepackage[T1]{fontenc} 
\usepackage[utf8]{inputenc}

\ifdefined\ifdraftdoc

\else
  \usepackage{ifdraft}
  \makeatletter
  \let\ifdraftdoc\if@draft
  \makeatother
\fi

\usepackage{needspace}

\usepackage{multirow}
\usepackage[leqno]{amsmath}
\allowdisplaybreaks[4]
\usepackage{amsthm}
\usepackage{amssymb}
\usepackage{amsfonts}

\renewcommand{\qedsymbol}{$\blacksquare$}
\newcommand{\openqedsymbol}{$\square$}
\newcommand{\exercisesymbol}{$\lozenge$}
\newcommand{\obvioussymbol}{$\blacklozenge$}
\newcommand{\openproof}{\renewcommand{\qedsymbol}{\openqedsymbol}}
\newcommand{\exerproof}{\renewcommand{\qedsymbol}{\exercisesymbol}}
\newcommand{\obviousproof}{\renewcommand{\qedsymbol}{\obvioussymbol}}

\usepackage{mathrsfs}
\DeclareMathAlphabet{\mathcal}{OMS}{cmsy}{m}{n}
\usepackage{stmaryrd}

\usepackage{mathtools}

\usepackage[british]{babel}

\usepackage[autostyle]{csquotes}
\usepackage{hyphenat}

\usepackage{etoolbox}
\usepackage{xstring}

\usepackage[shortlabels]{enumitem}

\usepackage[dvipsnames,hyperref]{xcolor}

\usepackage{tikz}
\usepackage{tikz-cd}

\usetikzlibrary{arrows, shapes.geometric}
\tikzset{baseline=(current bounding box.center)}

\let\LiningNumbers\relax

\DeclareTextFontCommand{\textrmup}{\rmfamily\mdseries\upshape}
\DeclareTextFontCommand{\textbfup}{\rmfamily\bfseries\upshape}
\DeclareTextFontCommand{\textttup}{\ttfamily\mdseries\upshape}
\DeclareTextFontCommand{\textsfup}{\sffamily\mdseries\upshape}

\setlist[enumerate]{font=\upshape}

\DeclareMathAlphabet{\mathpzc}{OT1}{pzc}{m}{it}

\usepackage{microtype}

\makeatletter
\newcommand*{\noprelistbreak}{\@nobreaktrue\nopagebreak}
\makeatother

\usepackage{xspace}

\newcommand{\hairspace}{\ifmmode\mskip1mu\else\kern0.08em\fi}

\makeatletter
\def\abbrdot{\@ifnextchar.{}{.\@\xspace}}
\makeatother

\newcommand{\ie}{i.e\abbrdot}
\newcommand{\resp}{resp\abbrdot}

\newcommand{\opcit}{op.\@ cit\abbrdot}

\newcommand{\Chap}{Ch.\@~}

\newcommand{\Sect}{\S\hairspace}
\newcommand{\Sects}{\S\S\hairspace}

\newcommand{\strong}[1]{\textbfup{#1}}

\renewcommand{\implies}{\ifmmode\Longrightarrow\else$\Rightarrow$\expandafter\xspace\fi}
\renewcommand{\iff}{\ifmmode\Longleftrightarrow\else$\Leftrightarrow$\expandafter\xspace\fi}

\makeatletter
\let\internal@prime\prime
\renewcommand{\prime}{\ifmmode\internal@prime\else$\sp{\internal@prime}$\expandafter\xspace\fi}
\makeatother

\ifdefined\nequiv

\else
  \def\nequiv{\not\equiv}
\fi

\newcommand{\blank}{\mathord{-}}
\newcommand{\pblank}{\mathord{(-)}}

\newcommand{\embedinto}{\hookrightarrow}
\newcommand{\hoto}{\Rightarrow}

\DeclareMathSymbol{.}{\mathpunct}{letters}{"3A}
\DeclareMathSymbol{:}{\mathrel}{operators}{"3A}

\DeclareMathOperator{\ob}{ob}

\newcommand{\cotens}{\mathbin{\pitchfork}}

\DeclareMathOperator{\discr}{disc}

\let\prolim\varprojlim
\let\indlim\varinjlim

\newcommand{\id}{\mathrm{id}}

\AtBeginDocument{}

\newcommand*{\lr}[3]{\mathopen{}\left#1{#2}\right#3\mathclose{}}
\newcommand*{\lmr}[5]{\mathopen{}\left#1{#2}\,\middle#3\,{#4}\right#5\mathclose{}}

\newcommand*{\ordlr}[3]{\mathord{\left#1{#2}\right#3}}
\newcommand*{\argp}[1]{\ifstrempty{#1}{}{\lr({#1})}}
\newcommand*{\argptwo}[2]{\lr({#1, #2})}
\newcommand*{\argptwotwo}[4]{\lr({#1, #2; #3, #4})}

\newcommand*{\parens}[1]{\lr({#1})}

\newcommand*{\bracket}[1]{\lr[{#1}]}
\newcommand*{\tuple}[1]{\ordlr({#1})}

\newcommand*{\prodtuple}[1]{\lr<{#1}>}

\newcommand*{\setbuilder}[2]{\mathord{\lmr\{{#1}|{#2}\}}}
\newcommand*{\seqbuilder}[2]{\mathord{\lmr({#1}|{#2})}}

\ifdefined\fracslash

\else
  \def\fracslash{\mathbin{/}}
\fi
\ifdefined\divslash

\else
  \def\divslash{\mathbin{/}}
\fi

\makeatletter
\newcommand*{\optargp}{\new@ifnextchar\bgroup{\argp}{}}
\newcommand*{\optargptwo}{\new@ifnextchar\bgroup{\expandafter\optargptwo@first}{}}
\newcommand*{\optargptwo@first}[1]{\new@ifnextchar\bgroup{\argptwo{#1}}{\argp{#1}}}
\newcommand*{\optargptwotwo}{\new@ifnextchar\bgroup{\expandafter\optargptwotwo@first}{}}
\newcommand*{\optargptwotwo@first}[2]{\new@ifnextchar\bgroup{\argptwotwo{#1}{#2}}{\argptwo{#1}{#2}}}

\newcommand*{\sbtwo}[2]{\sb{\lr({{#1}, {#2}})}}

\newcommand*{\optsb}{\new@ifnextchar\bgroup{\sb}{}}
\newcommand*{\optsp}{\new@ifnextchar\bgroup{\sp}{}}
\newcommand*{\optsbtwo}{\new@ifnextchar\bgroup{\expandafter\optsbtwo@first}{}}
\newcommand*{\optsbtwo@first}[1]{\new@ifnextchar\bgroup{\sbtwo{#1}}{\sb{#1}}}

\newcommand*{\optspargp}[1][]{\sp{#1}\optargp}
\newcommand*{\optsbspargp}[1][]{\sb{#1}\optspargp}

\newcommand*{\set}[1]{\new@ifnextchar\bgroup{\setbuilder{#1}}{\ordlr\{{#1}\}}}
\newcommand*{\seq}[1]{\new@ifnextchar\bgroup{\seqbuilder{#1}}{\tuple{#1}}}
\makeatother

\ifdefined\hash

\else
  \def\hash{\#}
\fi

\ifdefined\powerset
  \undef\powerset
\else

\fi
\newcommand*{\powerset}[1][]{\mathscr{P}\ifstrempty{#1}{\hairspace}{\sb{#1}} \optargp}

\newcommand{\mbfs}{\mathbf{s}}

\newcommand*{\smashsp}[2]{{#1}\sp{\smash{#2}}}

\newcommand*{\op}[1]{\smashsp{#1}{\mathrm{op}}}

\makeatletter
\newcommand*{\ctchoice}[2]{%
  \ifdef{\ct@homstyle}%
  {#2}%
  {#1}}
\newcommand*{\ctchoicetwo}[5]{%
  \ifdef{\ct@homcatstyle}%
  {#2}{%
  \ifdef{\ct@transfstyle}%
  {#3}{%
  \ifdef{\ct@catstyle}%
  {#4}{%
  \ifdef{\ct@homstyle}%
  {#5}{%
  {#1}}}}}}
  
\newcommand*{\DefineCategory}[2]{\DeclareRobustCommand{#1}{\ctchoice{\mathbf{#2}}{\mathbf{#2}}}}

\newcommand*{\cat}[1]{%
  \begingroup%
  \def\ct@catstyle{cat}%
  #1%
  \endgroup}
\newcommand*{\bicat}[1]{%
  \begingroup
  \def\ct@bicatstyle{bicat}%
  #1%
  \endgroup}
\newcommand*{\Hom}[1][]{%
  \begingroup%
  \def\ct@homstyle{hom}%
  \ifstrempty{#1}{\mathrm{Hom}}{#1}%
  \endgroup%
  \optargptwo}
\newcommand*{\HHom}[1][]{%
  \begingroup%
  \def\ct@homcatstyle{homcat}%
  \ifstrempty{#1}{\mathbf{Hom}}{#1}%
  \endgroup
  \optargptwo}
\newcommand*{\hpty}[1][]{%
  \begingroup%
  \def\ct@transfstyle{transf}%
  \ifstrempty{#1}{\mathrm{Nat}}{#1}%
  \endgroup
  \optargptwotwo}

\makeatother

\DefineCategory{\Set}{Set}

\DefineCategory{\Simplex}{\Delta}
\DefineCategory{\SSet}{sSet}
\newcommand{\Kancx}[1][]{\ctchoice{\mathbf{Kan}\ifstrempty{#1}{}{\argp{#1}}}{\mbfs {#1}}}

\makeatletter
\newcommand*{\Sh@one}[1]{\ctchoice{\mathbf{Sh}\argp{#1}}{\mathbf{Sh}_{#1}}}
\newcommand*{\Sh@two}[2]{\ctchoice{\mathbf{Sh}\argp{#1, #2}}{\mathbf{Sh}_{\tuple{#1, #2}}}}

\newcommand*{\Sh}[1]{\new@ifnextchar\bgroup{\expandafter\Sh@two{#1}}{\Sh@one{#1}}}
\makeatother

\newcommand*{\Psh}[1]{\ctchoice{\mathbf{Psh}\argp{#1}}{\mathbf{Psh}_{#1}}}

\newcommand*{\SPsh}{\mbfs \Psh}
\newcommand*{\SSh}{\mbfs \Sh}

\newcommand*{\overcat}[2]{\mathord{#1 \sb{\mathord{\divslash} #2}}}

\newcommand*{\nv}[1][]{\mathrm{N}\sp{#1}\optargp}

\newcommand{\Kompakt}{\mathbf{K}\optsbspargp}

\newcommand{\Ex}[1][]{\mathrm{Ex}\sp{#1}\optargp}

\newcommand*{\xHom}[3][]{\ifstrempty{#1}{\ordlr[{{#2}, {#3}}]}{\ordlr[{{#2}, {#3}}]\sb{#1}}}

\newcommand*{\Func}[2]{\xHom{#1}{#2}}

\newcommand*{\kwlim}[3][]{\ordlr\lbrace{{#2}, {#3}}\rbrace\sp{#1}}

\makeatletter
\renewcommand*{\@makefnmark}{\hbox{\textsuperscript{\normalfont [\LiningNumbers\@thefnmark]}}}
\renewcommand*{\@makefntext}[1]{\parindent 1.0em\noindent\ifdefempty{\@thefnmark}{}{\hb@xt@-1.0em{\hss \normalfont [\LiningNumbers \@thefnmark]}\hspace{1.0em}}#1}
\@addtoreset{footnote}{chapter}
\newcommand{\footpar}[1]{\gdef\@thefnmark{}\@footnotetext{#1}}

\newcommand{\hangsecnum}{\def\@seccntformat##1{\llap{\csname the##1\endcsname\quad}}}
\makeatother

\usepackage[natbib,safeinputenc,backend=biber,style=authoryear-comp,bibstyle=authoryear]{biblatex}

\DeclareSortingScheme{mysort}{
  \sort{
    \field{presort}
  }
  \sort[final]{
    \field{sortkey}
  }
  \sort{
    \name{sortname}
    \name{author}
    \name{editor}
    \name{translator}
    \field{sorttitle}
    \field{title}
  }
  \sort{
    \field{sortyear}
    \field{year}
  }
  \sort{
    \field{month}
  }
  \sort{
    \field{day}
  }
  \sort{
    \field{journal}
  }
  \sort{
    \field[padside=left,padwidth=4,padchar=0]{volume}
    \literal{0000}
  }
  \sort{
    \field{series}
  }
  \sort{
    \field[padside=left,padwidth=4,padchar=0]{number}
    \literal{0000}
  }
  \sort{
    \field{sorttitle}
    \field{title}
  }
}

\usepackage[colorlinks,linkcolor=Blue,citecolor=Blue,urlcolor=blue,pdfpagelabels,pdffitwindow=false]{hyperref}

\usepackage{thmtools}

\declaretheorem[style=plain,parent=section,title=Theorem,refname=theorem,Refname=Theorem]{thm}
\declaretheorem[style=plain,sibling=thm,title=Proposition,refname=proposition,Refname=Proposition]{prop}
\declaretheorem[style=plain,sibling=thm,title=Lemma,refname=lemma,Refname=Lemma]{lem}
\declaretheorem[style=plain,sibling=thm,title=Corollary,refname=corollary,Refname=Corollary]{cor}

\declaretheorem[style=definition,sibling=thm,title=Definition,refname=definition,Refname=Definition]{dfn}

\declaretheorem[style=definition,sibling=thm,title=Example,refname=example,Refname=Example]{example}

\declaretheoremstyle[style=remark,headfont=\scshape]{remark}
\declaretheorem[style=remark,sibling=thm,title=Remark,refname=remark,Refname=Remark]{remark}

\declaretheorem[style=plain,numbered=no,title=Theorem]{thm*}
\declaretheorem[style=plain,numbered=no,title=Proposition]{prop*}
\declaretheorem[style=remark,numbered=no,title=Remark]{remark*}
\declaretheorem[style=definition,numbered=no,title=Example]{example*}

\makeatletter
\@counteralias{numpar}{thm}

\newcommand*{\makenumpar}[1][]{%
  \par%
  \refstepcounter{numpar}%
  {\normalfont{\bfseries\textparagraph}\hspace{0.5em}\thenumpar\ifstrempty{#1}{}{ ({\normalfont #1})}.} }
\makeatother

\AtBeginDocument{%
\hyphenation{pull-back pull-backs push-out push-outs co-pro-duct co-pro-ducts co-com-plete sub-ob-ject sub-ob-jects more-over quasi-com-pact de-gen-er-ate pro-vided Gro-then-dieck Law-vere mod-ule sub-al-ge-bra fi-brant co-fi-brant fi-brant-ly co-fi-brant-ly bi-cat-e-go-ry bi-cat-e-gor-ic-al bi-cat-e-gor-ic-al-ly qua-si-cat-e-go-ry qua-si-cat-e-go-ries qua-si-cat-e-go-ric-al}%
}

\renewcommand*{\bibnamedash}{}
\newcommand*{\authoryearpunct}{\hspace*{-\bibhang}\bibsentence}

\makeatletter

\appto\citesetup{\LiningNumbers}

\renewbibmacro*{cite:label}{%
  \iffieldundef{label}
    {\printtext[bibhyperref]{\printfield[citetitle]{labeltitle}}}
    {\printtext[bibhyperref]{\printfield{label}}}}

\renewbibmacro*{date+extrayear}{%
  \iffieldundef{shorthand}
    {\iffieldundef{labelyear}
      {}
      {\printtext[brackets]{\printfield{labelyear}\printfield{extrayear}}}\addspace}
    {\printtext[brackets]{\printfield{shorthand}}}}

\renewbibmacro*{author}{%
  \ifboolexpr{
    test \ifuseauthor
    and
    not test {\ifnameundef{author}}
  }
    {\usebibmacro{bbx:dashcheck}
       {\bibnamedash}
       {\usebibmacro{bbx:savehash}%
        \printnames{author}%
    \newline\setunit{\authoryearpunct}}%
     \iffieldundef{authortype}
       {}
       {\usebibmacro{authorstrg}%
    \setunit{\addspace}}}%
    {\global\undef\bbx@lasthash
     \usebibmacro{labeltitle}%
     \setunit*{\addspace}}%
  \usebibmacro{date+extrayear}}

\renewbibmacro*{bbx:editor}[1]{%
  \ifboolexpr{
    test \ifuseeditor
    and
    not test {\ifnameundef{editor}}
  }
    {\usebibmacro{bbx:dashcheck}
       {\bibnamedash}
       {\printnames{editor}%
    \newline\setunit{\authoryearpunct}%
    \usebibmacro{bbx:savehash}}%
     \usebibmacro{#1}%
     \clearname{editor}%
     \setunit{\addspace}}%
    {\global\undef\bbx@lasthash
     \usebibmacro{labeltitle}%
     \setunit*{\addspace}}%
  \usebibmacro{date+extrayear}}

\renewbibmacro*{bbx:translator}[1]{%
  \ifboolexpr{
    test \ifusetranslator
    and
    not test {\ifnameundef{translator}}
  }
    {\usebibmacro{bbx:dashcheck}
       {\bibnamedash}
       {\printnames{translator}%
    \newline\setunit{\authoryearpunct}%
    \usebibmacro{bbx:savehash}}%
     \usebibmacro{translator+othersstrg}%
     \clearname{translator}%
     \setunit{\addspace}}%
    {\global\undef\bbx@lasthash
     \usebibmacro{labeltitle}%
     \setunit*{\addspace}}%
  \usebibmacro{date+extrayear}}

\makeatother

\appto\bibsetup{\raggedright}
\appto\bibfont{\LiningNumbers\small}

\usepackage{titling}
\usepackage{fancyhdr}

\hangsecnum

\title{Internal and local homotopy theory}
\author{Zhen~Lin Low}
\date{1 May 2014}

\begin{document}

\maketitle
\footpar{Department of Pure Mathematics and Mathematical Statistics, University of Cambridge, Cambridge, UK. \textsc{E-mail address}: \texttt{Z.L.Low@dpmms.cam.ac.uk}}

\begin{abstract}
There is a well-established homotopy theory of simplicial objects in a Grothendieck topos, and folklore says that the weak equivalences are axiomatisable in the geometric fragment of $L_{\omega_1, \omega}$. We show that it is in fact a theory of presheaf type, \ie classified by a presheaf topos. As a corollary, we obtain a new proof of the fact that the local Kan fibrations of simplicial presheaves that are local weak homotopy equivalences are precisely the morphisms with the expected local lifting property.
\end{abstract}

\section*{Introduction}

It has long been suggested that the proper setting for non-abelian sheaf cohomology  should be some kind of category of ``sheaves'' of homotopy types.  The first concrete construction of such a category was perhaps given by \citet{Brown:1973}: in \opcit, he defines the notion of a `category of fibrant objects' and mentions that sheaves (on a topological space) of Kan complexes constitute such a category. As with the definition of `sheaf of local rings', we must interpret Kan's extension condition \emph{locally}; and in the case of sheaves on a topological space, `locally' means the same thing as `stalkwise'.

The logical next step was taken by \citet{Jardine:1986}, who showed that there is a category of fibrant objects consisting of locally fibrant simplicial sheaves on a Grothendieck site. In some sense, this development was anticipated by \citet{VanOsdol:1977}, but Jardine avoided using the embedding theorem of \citet{Barr:1971} and instead gave direct proofs of the required properties. \citet{Jardine:1987} then observed that the resulting homotopy theory could be transferred to the category of locally fibrant simplicial \emph{presheaves}. This, as \citet{DHI:2004} explained, should be regarded as a step toward the notion of homotopy-coherent ``sheaf'' of spaces.

In this paper, we re-examine the above-mentioned results through the lens of category-theoretic logic. This will hopefully shed some light on the status of the homotopy theory of simplicial sets in constructive mathematics. The structure of the paper is as follows:
\begin{itemize}
\item In \Sect 1, we show that the category of internal Kan complexes in any regular category is a category of fibrant objects (in the sense of Brown, as mentioned above), generalising the result of \citet{Jardine:1986}.

\item In \Sect 2, we construct an internal version of Kan's $\Ex[\infty]$ functor and use it to define `weak equivalence' for general simplicial objects in a $\sigma$-pretopos.

\item In \Sect 3, we prove that weak equivalences of simplicial sheaves are classified by a presheaf topos. At first glance, this may seem to be a simple corollary of a theorem of \citet{Raptis-Rosicky:2014}, but there some subtleties in showing that the morphisms classified by the obvious construction do indeed coincide with the weak equivalences previously defined.

\item In \Sect 4, we use the classifying topos for weak homotopy equivalences to obtain some additional results: in particular, we will see that a morphism is an internal Kan fibration and a internal weak homotopy equivalence if and only if it has the internal lifting property analogous to the well-known one for trivial Kan fibrations of simplicial sets.

\item In \Sect 5, we apply the above results to Jardine's local homotopy theory of simplicial presheaves.
\end{itemize}
We will often use a completeness theorem of one kind or another to transfer results from the classical homotopy theory of simplicial sets to more general settings. Proofs using such techniques will be marked with $\lozenge$. Familiarity with basic topos theory is assumed from \Sect 3 onward; standard references include \citep[\Chap II, VII, and X]{MLM} and \citep[Part C, \Sects 2.1--2.3; Part D, \Sect 3.1]{Johnstone:2002b}.

Astute readers may notice that model structures do not make any appearance here. This is partly because weak factorisation systems do not have the expected properties in constructive settings: indeed, the well-known fact that the left class of a weak factorisation system is closed under all coproducts turns out to be equivalent to the axiom of choice. (Consider the weak factorisation system cofibrantly generated by $0 \to 1$.) That is \emph{not} to say that model structures cannot exist: after all, model categories of simplicial (pre)sheaves have been constructed by \citet{Brown-Gersten:1973}, \citet{Joyal-to-Grothendieck}, and \citet{Jardine:1987}. Rather, what the observation suggests is that the classical Kan--Quillen model structure cannot be \emph{internalised}\hairspace ---\hairspace and at any rate, an \emph{internal} model structure is not what we want.

\subsection*{Acknowledgements}

The methods used in this paper were inspired in no small part by the work of Tibor Beke \parencite*{Beke:1998,Beke:2000}. Many thanks also to Peter Johnstone: first, for taking the present writer as a Ph.D. student, and second, for invaluable assistance in navigating the literature on topos theory. Discussions with Achilleas Kryftis and Mike Shulman helped clarify various points in the exposition.

The author gratefully acknowledges financial support from the Cambridge Commonwealth, European and International Trust and the Department of Pure Mathematics and Mathematical Statistics.

\section{Internal homotopy theory}

We reformulate some of the definitions of \citet{Jardine:1986} in site-free language. These concepts make sense and behave well in the context of regular categories, a large class that includes both toposes and abelian categories. Readers unfamiliar with regular categories may refer to \Sect 1.3 of \citep[Part A]{Johnstone:2002a} or Chapter 2 of \citep{Borceux:1994b}.

\begin{dfn}
A \strong{weak pullback square} in a regular category $\mathcal{S}$ is a commutative diagram in $\mathcal{S}$, say
\[
\begin{tikzcd}
Z \dar \rar &
X \dar \\
W \rar &
Y
\end{tikzcd}
\]
such that the comparison morphism $Z \to W \times_Y X$ is a regular epimorphism.
\end{dfn}

\begin{lem}
\label{lem:generalities:lifting.properties.and.pullback.products}
Let $f : X \to Y$ and $g : Z \to W$ be morphisms in a locally small category $\mathcal{C}$. Consider the commutative diagram in $\cat{\Set}$ shown below:
\[
\begin{tikzcd}
\Hom[\mathcal{C}]{W}{X} \dar[swap]{g^*} \rar{f_*} &
\Hom[\mathcal{C}]{W}{Y} \dar{g^*} \\
\Hom[\mathcal{C}]{Z}{X} \rar[swap]{f_*} &
\Hom[\mathcal{C}]{Z}{Y}
\end{tikzcd}
\]
\begin{enumerate}[(i)]
\item The diagram is a weak pullback square if and only if $g$ has the left lifting property with respect to $f$.

\item The diagram is a pullback square if and only if $g$ is left orthogonal to $f$.
\end{enumerate}
\end{lem}
\begin{proof} \obviousproof
This is just a diagram chase.
\end{proof}

\needspace{3.0\baselineskip}
It will be convenient to borrow a concept from enriched category theory:

\begin{dfn}
Let $\mathcal{C}$ be a locally small category and let $W : \mathcal{J} \to \cat{\Set}$ be a functor. A \strong{$W$-weighted limit} of a diagram $X : \mathcal{J} \to \mathcal{C}$ is an object $\kwlim{W}{X}$ equipped with bijections
\[
\Hom[\Func{\mathcal{J}}{\cat{\Set}}]{W}{\Hom[\mathcal{C}]{T}{X}} \cong \Hom[\mathcal{C}]{T}{\kwlim{W}{X}}
\]
that are natural in $T$.
\end{dfn}

\begin{remark}
Let $\mathcal{C}$ be a locally small category, let $W : \mathcal{J} \to \cat{\Set}$ be a functor, and let $X : \mathcal{J} \to \mathcal{C}$ be a diagram. If $W$ is a representable functor, say $W = \Hom[\mathcal{J}]{j}{\blank}$, then the weighted limit $\kwlim{W}{X}$ exists and is naturally isomorphic to the object $X j$. More generally, if $W \cong \prolim_\mathcal{I} \Hom[\mathcal{J}]{F}{\blank}$ for some diagram $F : \mathcal{I} \to \mathcal{J}$ and $\mathcal{C}$ has limits for diagrams of shape $\mathcal{I}$, then the weighted limit $\kwlim{W}{X}$ exists and is naturally isomorphic to the limit $\prolim_\mathcal{I} X F$.
\end{remark}

\begin{example}
Let $\mathcal{C}$ be a locally small category with finite limits. Given a simplicial object $X : \op{\cat{\Simplex}} \to \mathcal{C}$, the weighted limit $\kwlim{\partial \Delta^n}{X}$ exists and can be identified with the $n$-th matching object of $X$. Thus, we may regard weighted limits as generalised matching objects.
\end{example}

Given a category $\mathcal{C}$, let $\mbfs \mathcal{C}$ be the category of simplicial objects in $\mathcal{C}$.

\begin{dfn}
An \strong{internal Kan fibration} (\resp \strong{internal trivial Kan fibration}) in a regular category $\mathcal{S}$ is a morphism $p : X \to Y$ in $\mbfs \mathcal{S}$ with the following property: 
\begin{itemize}
\item If $i : Z \to W$ is a horn inclusion $\Lambda^n_k \embedinto \Delta^n$ (\resp boundary inclusion $\partial \Delta^n \embedinto \Delta^n$), then the square in the diagram below is a weak pullback square in $\mathcal{S}$:
\[
\begin{tikzcd}
\kwlim{W}{X} \dar[swap]{\kwlim{i}{X}} \rar{\kwlim{W}{p}} &
\kwlim{W}{Y} \dar{\kwlim{i}{Y}} \\
\kwlim{Z}{X} \rar[swap]{\kwlim{Z}{p}} &
\kwlim{Z}{Y}
\end{tikzcd}
\]
\end{itemize}
\end{dfn}

\begin{remark}
\label{rem:hotoposes:internal.Kan.fibrations.in.Set}
If $\mathcal{S} = \cat{\Set}$, then an internal Kan fibration is just a Kan fibration in the usual sense, by \autoref{lem:generalities:lifting.properties.and.pullback.products}. If $\mathcal{S} = \Func{\op{\mathcal{C}}}{\cat{\Set}}$ for a small category $\mathcal{C}$, then an internal Kan fibration in $\mathcal{S}$ is the same thing as a componentwise Kan fibration, because limits and colimits in $\Func{\op{\mathcal{C}}}{\cat{\Set}}$ are computed componentwise. And if $\mathcal{S} = \cat{\Sh{\mathcal{C}}{J}}$ for a small Grothendieck site $\tuple{\mathcal{C}, J}$, then an internal Kan fibration in $\mathcal{S}$ is precisely a fibration of simplicial sheaves in the sense of \citet{Jardine:1986}.
\end{remark}

\begin{dfn}
An \strong{internal Kan complex} in a regular category $\mathcal{S}$ is an object $X$ in $\mbfs \mathcal{S}$ such that the unique morphism $X \to 1$ in $\mbfs \mathcal{S}$ is an internal Kan fibration. We write $\cat{\Kancx[\mathcal{S}]}$ for the full subcategory of $\mbfs \mathcal{S}$ spanned by the internal Kan complexes in $\mathcal{S}$.
\end{dfn}

For our purposes, it is most convenient to define weak homotopy equivalences in terms of an internal lifting property. Let $D^{n+1}$ be the simplicial set defined by the following pushout diagram in $\cat{\SSet}$,
\[
\begin{tikzcd}
\partial \Delta^n \times \Delta^1 \dar \rar[hookrightarrow] &
\Delta^n \times \Delta^1 \dar{q} \\
\partial \Delta^n \rar &
D^{n+1}
\end{tikzcd}
\]
where $\partial \Delta^n \times \Delta^1 \to \partial \Delta^n$ is the projection and $
\partial \Delta^n \times \Delta^1 \embedinto \Delta^n \times \Delta^1$ is induced by the boundary inclusion $\partial \Delta^n \embedinto \Delta^n$. Let $j_0, j_1 : \Delta^n \to D^{n+1}$ be the morphisms obtained by composing with $q : \Delta^n \times \Delta^1 \to D^{n+1}$ the two morphisms $\Delta^n \to \Delta^n \times \Delta^1$ induced by the two vertex inclusions $\Delta^0 \to \Delta^1$.

\begin{dfn}
\label{dfn:hotoposes:DI.whe}
A \strong{Dugger--Isaksen weak equivalence} in a regular category $\mathcal{S}$ is a morphism $f : X \to Y$ in $\cat{\Kancx[\mathcal{S}]}$ such that the morphism
\[
\kwlim{\Delta^n}{X} \times_{\kwlim{\Delta^n}{Y}} \kwlim{D^{n+1}}{Y} \longrightarrow \kwlim{\partial \Delta^n}{X} \times_{\kwlim{\partial \Delta^n}{Y}} \kwlim{\Delta^n}{Y}
\]
induced by the commutative diagram
\[
\begin{tikzcd}[column sep=3.0ex, row sep=3.0ex]
\bullet \arrow{dd} \drar[dashed] \arrow{rr} &&
\kwlim{D^{n+1}}{Y} \arrow[near end]{dd}{\kwlim{j_1}{Y}} \drar{\kwlim{j_0}{Y}} \\
&
\bullet \arrow[crossing over]{dd} \arrow[crossing over]{rr} &&
\kwlim{\Delta^n}{Y} \arrow{dd} \\
\kwlim{\Delta^n}{X} \drar \arrow[swap, near end]{rr}{\kwlim{\Delta^n}{f}} &&
\kwlim{\Delta^n}{Y} \drar \\
&
\kwlim{\partial \Delta^n}{X} \arrow[swap]{rr}{\kwlim{\partial \Delta^n}{f}} &&
\kwlim{\partial \Delta^n}{Y}
\end{tikzcd}
\]
is a regular epimorphism in $\mathcal{S}$.
\end{dfn}

\begin{remark}
In the case $\mathcal{S} = \cat{\Set}$, a Dugger--Isaksen weak equivalence is a morphism $f : X \to Y$ in $\cat{\Kancx}$ with the following generalised right lifting property: given a commutative diagram of the form below,
\[
\begin{tikzcd}
\partial \Delta^n \dar[hookrightarrow] \rar{\partial x} &
X \dar{f} \\
\Delta^n \rar[swap]{y} &
Y
\end{tikzcd}
\]
there exist $x : \Delta^n \to X$ and $h : D^{n+1} \to Y$ making the following diagram commute:
\[
\begin{tikzcd}
\partial \Delta^n \dar[hookrightarrow] \rar[hookrightarrow] \arrow[bend left=25]{rr}{\partial x} &
\Delta^n \dar{j_1} \rar[dashed]{x} &
X \dar{f} \\
\Delta^n \rar[swap]{j_0} \arrow[swap, bend right=25]{rr}{y} &
D^{n+1} \rar[swap, dashed]{h} &
Y
\end{tikzcd}
\]
In other words, every $n$-simplex of $Y$ whose boundary admits a lift to $X$ is homotopic (rel.\@ $\partial \Delta^n$) to an $n$-simplex in the image of $f : X \to Y$.
\end{remark}

\begin{thm}
\label{thm:hotoposes:Dugger-Isaksen.weak.equivalences}
Let $f : X \to Y$ be a morphism of Kan complexes in $\cat{\Set}$. The following are equivalent:
\begin{enumerate}[(i)]
\item $f : X \to Y$ is a (weak) homotopy equivalence.

\item $f : X \to Y$ is a Dugger--Isaksen weak equivalence.
\end{enumerate}
\end{thm}
\begin{proof} \openproof
See Proposition 4.1 in \citep{Dugger-Isaksen:2004b}.
\end{proof}

It is clear that the property of being an internal Kan fibration, internal trivial Kan fibration, or Dugger--Isaksen weak equivalence is preserved by regular functors and is reflected by conservative regular functors. To transfer results from the classical setting of Kan complexes in $\cat{\Set}$ to the general setting of internal Kan complexes in any regular category, we follow \citet{VanOsdol:1977} in using the classical completeness theorem for regular logic:

\begin{thm}[Barr]
\label{thm:generalities:classical.completeness.for.regular.categories}
For each small regular category $\mathcal{C}$, there exist a set $B$ and a conservative regular functor $\mathcal{C} \to \cat{\Set}^B$.
\end{thm}
\begin{proof} \openproof
See Theorem 1.6 in \citep[\Chap III]{Barr:1971} or Corollary 1.5.4 in \citep[Part D]{Johnstone:2002b}.
\end{proof}

\begin{dfn}
Let $Z$ be a simplicial set and let $X$ be a simplicial object in a locally small category $\mathcal{C}$. A \strong{cotensor product} $Z \cotens X$ is a simplicial object in $\mathcal{C}$ equipped with bijections
\[
\Hom[\SSet]{\Delta^n \times Z}{\Hom[\mathcal{C}]{T}{X}} \cong \Hom[\mathcal{C}]{T}{\parens{Z \cotens X}_n}
\]
that are natural in $n$ and $T$.
\end{dfn}

\begin{remark}
If $\mathcal{C}$ has finite limits and $Z$ is a finite simplicial set, then $Z \cotens X$ exists for all simplicial objects $X$. Furthermore, there is a simplicial enrichment of $\mbfs \mathcal{C}$ such that the cotensor products defined above have the well-known enriched universal property.
\end{remark}

\begin{lem}
\label{lem:hotoposes:cotensors.and.internal.Kan.complexes}
Let $X$ be an internal Kan complex in a regular category $\mathcal{S}$.
\begin{enumerate}[(i)]
\item For all finite simplicial sets $Z$, the cotensor product $Z \cotens X$ is also an internal Kan complex.

\item If $g : Z \to W$ is a monomorphism of finite simplicial sets, then the induced morphism $g \cotens \id_X : W \cotens X \to Z \cotens X$ is an internal Kan fibration.

\item If $g : Z \to W$ is an anodyne extension of finite simplicial sets, then the induced morphism $g \cotens \id_X : W \cotens X \to Z \cotens X$ is an internal trivial Kan fibration.
\end{enumerate}
\end{lem}
\begin{proof} \exerproof
Use \autoref{thm:generalities:classical.completeness.for.regular.categories} to reduce to the case where $\mathcal{S} = \cat{\Set}$.
\end{proof}

\begin{cor}[Internal path spaces]
\label{cor:hotoposes:internal.path.spaces}
Let $\mathcal{S}$ be a regular category and let $X$ be an internal Kan complex in $\mathcal{S}$.
\begin{enumerate}[(i)]
\item $\Delta^1 \cotens X$ is an internal Kan complex in $\mathcal{S}$.

\item The morphism $\Delta^1 \cotens X \to \partial \Delta^1 \cotens X$ induced by the boundary inclusion $\partial \Delta^1 \embedinto \Delta^1$ is an internal Kan fibration in $\mathcal{S}$.

\item The morphisms $\Delta^1 \cotens X \to \Delta^0 \cotens X$ induced by the two vertex inclusions $\Delta^0 \to \Delta^1$ are internal trivial Kan fibrations in $\mathcal{S}$. \qed
\end{enumerate}
\end{cor}
%
%
%

\begin{thm}
\label{thm:hotoposes:category.of.internally.fibrant.objects}
Let $\mathcal{S}$ be a regular category. 
\begin{enumerate}[(i)]
\item The class of Dugger--Isaksen weak equivalences has the 2-out-of-3 property and is closed under retracts in $\cat{\Kancx[\mathcal{S}]}$.

\item An internal Kan fibration of internal Kan complexes in $\mathcal{S}$ is an internal trivial Kan fibration if and only if it is also a Dugger--Isaksen weak equivalence.

\item $\cat{\Kancx[\mathcal{S}]}$ is a category of fibrant objects\footnote{--- in the sense of \citet{Brown:1973}.} where the fibrations are the internal Kan fibrations and the weak equivalences are the Dugger--Isaksen weak equivalences.
\end{enumerate}
\end{thm}
\begin{proof} \exerproof
Use \autoref{thm:generalities:classical.completeness.for.regular.categories} to reduce to the case where $\mathcal{S} = \cat{\Set}$.
\end{proof}

We have so far avoided discussing homotopy groups for internal Kan complexes. This is because they need not exist in a general regular category: clearly, if we are able to take quotients of internal equivalence relations, then we can construct $\pi_0$; and conversely, because internal equivalence relations define internal Kan complexes, being able to construct $\pi_0$ implies we can take quotients of internal equivalence relations. This suggests that the right setting for these constructions is an effective regular category.\footnote{--- also known as `exact category in the sense of Barr'.}

\begin{prop}
\needspace{3.0\baselineskip}
Let $\mathcal{S}$ be an effective regular category and let $\pblank_0 : \cat{\Kancx[\mathcal{S}]} \to \mathcal{S}$ be the functor $X \mapsto X_0$.
\begin{enumerate}[(i)]
\item The functor $\pblank_0 : \cat{\Kancx[\mathcal{S}]} \to \mathcal{S}$ has a left adjoint, ${\discr} : \mathcal{S} \to \cat{\Kancx[\mathcal{S}]}$.

\item The functor ${\discr} : \mathcal{S} \to \cat{\Kancx[\mathcal{S}]}$ has a left adjoint, $\pi_0 : \cat{\Kancx[\mathcal{S}]} \to \mathcal{S}$.

\item Regular functors (between effective regular categories) commute with $\pi_0$.
\end{enumerate}
\end{prop}
\begin{proof}
(i). Let $\discr X$ be the constant simplicial object on $X$. Since the face and degeneracy operators of $\discr X$ are isomorphisms, the morphisms $\kwlim{\Delta^n}{\discr X} \to \kwlim{\Lambda^n_k}{\discr X}$ induced by the horn inclusions $\Lambda^n_k \embedinto \Delta^n$ must also be isomorphisms (and regular epimorphisms \emph{a fortiori}).

\bigskip\noindent
(ii). We define the functor $\pi_0 : \cat{\Kancx[\mathcal{S}]} \to \mathcal{S}$ by the following coequaliser diagram:
\[
\begin{tikzcd}
X_1 
\rar[transform canvas={yshift=0.75ex}]{d_1} 
\rar[transform canvas={yshift=-0.75ex}, swap]{d_0} &
X_0 
\rar &
\pi_0 X
\end{tikzcd}
\]
Such coequalisers exist, because the image of $\prodtuple{d_1, d_0} : X_1 \to X_0 \times X_0$ is an internal equivalence relation in $\mathcal{S}$. It is straightforward to verify that $\pi_0 : \cat{\Kancx[\mathcal{S}]} \to \mathcal{S}$ is indeed a left adjoint for ${\discr} : \mathcal{S} \to \cat{\Kancx[\mathcal{S}]}$.

\bigskip\noindent
(iii). Since $\pi_0$ is the quotient of an internal equivalence relation constructed by image factorisation, it is preserved by all regular functors.
\end{proof}

\begin{dfn}
Let $n$ be a positive integer and let $X$ be an internal Kan complex in an effective regular category $\mathcal{S}$. 
\begin{itemize}
\item The \strong{internal based $n$-loop fibration} on $X$ is the internal Kan fibration $\Omega^n \argp{X} \to X$ defined by the following pullback diagram in $\cat{\Kancx[\mathcal{S}]}$,
\[
\begin{tikzcd}
\Omega^n \argp{X} \dar \rar &
\Delta^n \cotens X \dar \\
X \rar &
\partial \Delta^n \cotens X
\end{tikzcd}
\]
where $\Delta^n \cotens X \to \partial \Delta^n \cotens X$ is the internal Kan fibration induced by the boundary inclusion $\partial \Delta^n \embedinto \Delta^n$ and $X \to \partial \Delta^n \cotens X$ is the morphism induced by $\partial \Delta^n \to \Delta^0$. 

\item Let $x$ be a morphism $\discr T \to X$ in $\cat{\Kancx[\mathcal{S}]}$. The \strong{internal based $n$-loop space} of $\tuple{X, x}$ is the internal Kan complex $\Omega^n \argp{X, x}$ in $\overcat{\mathcal{S}}{T}$ defined by the following pullback diagram in $\cat{\Kancx[\mathcal{S}]}$:
\[
\begin{tikzcd}
\Omega^n \argp{X, x} \dar \rar &
\Omega^n \argp{X} \dar \\
\discr T \rar[swap]{x} &
X
\end{tikzcd}
\]
The \strong{internal $n$-th homotopy group} of $\tuple{X, x}$ is the object $\pi_n \argp{X, x} = \pi_0 \Omega^n \argp{X, x}$ in $\overcat{\mathcal{S}}{T}$.
\end{itemize}
\end{dfn}

\begin{remark}
It is clear that the above constructions are functorial. Moreover, $\pi_n \argp{X, x}$ admits a natural internal group structure (in $\overcat{\mathcal{S}}{T}$); but we do not need this fact.
\end{remark}

\begin{thm}
\label{thm:hotoposes:internal.whe.and.internal.homotopy.groups}
Let $\mathcal{S}$ be an effective regular category and let $f : X \to Y$ be a morphism in $\cat{\Kancx[\mathcal{S}]}$. The following are equivalent:
\begin{enumerate}[(i)]
\item $f : X \to Y$ is a Dugger--Isaksen weak equivalence.

\item $\pi_0 f : \pi_0 X \to \pi_0 Y$ is an isomorphism in $\mathcal{S}$ and, for all positive integers $n$, all objects $T$ in $\mathcal{S}$, and all morphisms $x : \discr T \to X$ in $\cat{\Kancx[\mathcal{S}]}$, $\pi_n f : \pi_n \argp{X, x} \to \pi_n \argp{Y, f \circ x}$ is an isomorphism in $\overcat{\mathcal{S}}{T}$.

\item $\pi_0 f : \pi_0 X \to \pi_0 Y$ is an isomorphism in $\mathcal{S}$ and, for all positive integers $n$, $\pi_n f : \pi_n \argp{X, \bar{x}} \to \pi_n \argp{Y, f \circ \bar{x}}$ is an isomorphism in $\overcat{\mathcal{S}}{X_0}$, where $\bar{x} : \discr X_0 \to X$ is the component of the adjunction counit.
\end{enumerate}
\end{thm}
\begin{proof} \exerproof
Use theorems \ref{thm:hotoposes:Dugger-Isaksen.weak.equivalences} and \ref{thm:generalities:classical.completeness.for.regular.categories}.
\end{proof}

\begin{cor}
Let $\tuple{\mathcal{C}, J}$ be a small Grothendieck site and let $\mathcal{S} = \cat{\Sh{\mathcal{C}}{J}}$. Then the weak equivalences in $\Kancx[\mathcal{S}]$ are precisely the weak equivalences in the sense of \citet{Jardine:1986}, \ie the morphisms that induce isomorphisms of $\pi_0$ and all sheaves of homotopy groups. \qed
\end{cor}

\begin{remark}
It should be emphasised that the weak equivalences in $\Kancx[\mathcal{S}]$ really are \emph{weak} equivalences: \citet{Bezem-Coquand:2013} have constructed a pair of internal Kan complexes in the presheaf topos $\mathcal{S} = \Func{\mathbf{3}}{\cat{\Set}}$ that are weakly homotopy equivalent (by virtue of being the two canonical fibres of an internal Kan fibration over the internal version of $\Delta^1$) but not homotopy equivalent.
\end{remark}

\section{Internal fibrant replacement}

Let $P^n$ be the partially ordered set of non-empty subsets of $\bracket{n}$ and, for each monotone map $f : \bracket{n} \to \bracket{m}$, let $f_* : P^n \to P^m$ be the map induced by taking images. Taking nerves, this defines a functor $\nv{P^{\bullet}} : \cat{\Simplex} \to \cat{\SSet}$. Note that there is a natural surjective monotone map ${\max} : P^n \to \bracket{n}$, each with a canonical (but not natural!) splitting, so we get a natural transformation $\nv{\max} : \nv{P^{\bullet}} \hoto \Delta^{\bullet}$ whose components are split epimorphisms.

\begin{dfn}
Let $\mathcal{S}$ be a locally small category with finite limits. An \strong{extension} of a simplicial object $X$ in $\mathcal{S}$ is a simplicial object $\Ex{X}$ equipped with bijections
\[
\Hom[\mathcal{S}]{T}{\Ex{X}_n} \cong \Ex{\Hom[\mathcal{S}]{T}{X}}_n = \Hom[\SSet]{\nv{P^n}}{\Hom[\mathcal{S}]{T}{X}}
\]
that are natural in both $n$ and $T$. The \strong{canonical embedding} $i_X : X \to \Ex{X}$ is the one induced by the natural transformation $\nv{\max} : \nv{P^{\bullet}} \hoto \Delta^{\bullet}$.
\end{dfn}

\begin{remark}
In other words, $\Ex{X}_n$ is the weighted limit $\kwlim{\nv{P^n}}{X}$; in particular, $\Ex{X}$ exists because $\mathcal{S}$ has finite weighted limits and each $\nv{P^n}$ is a finite simplicial set. Note that $\Ex{X}$ is unique up to unique isomorphism, so we obtain a functor $\Ex : \mbfs \mathcal{S} \to \mbfs \mathcal{S}$; and the canonical embeddings constitute a natural transformation $i : \id_{\mbfs \mathcal{S}} \hoto \Ex$.
\end{remark}

\begin{lem}
\label{lem:hotoposes:horn.fillers.and.Ex.for.internal.simplicial.objects}
Let $X$ be a simplicial object in a regular category $\mathcal{S}$. Consider the following pullback diagram in $\mathcal{S}$,
\[
\begin{tikzcd}
\kwlim{\Lambda^n_k}{\Ex{X}} \times_{\kwlim{\Lambda^n_k}{\Ex[2]{X}}} \kwlim{\Delta^n}{\Ex[2]{X}} \dar \rar &
\kwlim{\Delta^n}{\Ex[2]{X}} \dar \\
\kwlim{\Lambda^n_k}{\Ex{X}} \rar[swap]{\kwlim{\Lambda^n_k}{i_{\Ex{X}}}} &
\kwlim{\Lambda^n_k}{\Ex[2]{X}}
\end{tikzcd}
\]
where the morphism $\kwlim{\Delta^n}{\Ex[2]{X}} \to \kwlim{\Lambda^n_k}{\Ex[2]{X}}$ is induced by the horn inclusion $\Lambda^n_k \embedinto \Delta^n$. Then,
\[
\kwlim{\Lambda^n_k}{\Ex{X}} \times_{\kwlim{\Lambda^n_k}{\Ex[2]{X}}} \kwlim{\Delta^n}{\Ex[2]{X}} \to \kwlim{\Lambda^n_k}{\Ex{X}}
\]
is a regular epimorphism in $\mathcal{S}$.
\end{lem}
\begin{proof} \exerproof
In the case where $\mathcal{S} = \cat{\Set}$, the claim is a reformulation of a well-known fact,\footnote{See Lemma 3.2 in \citep{Kan:1957a} or Lemma 4.7 in \citep[\Chap III]{GJ}.} namely the fact that every morphism $\Lambda^n_k \to \Ex{X}$ in $\cat{\SSet}$ fits into a commutative diagram of the form below:
\[
\begin{tikzcd}
\Lambda^n_k \dar[hookrightarrow] \rar &
\Ex{X} \dar{i_{\Ex{X}}} \\
\Delta^n \rar[dashed] &
\Ex[2]{X}
\end{tikzcd}
\]
For the general case, we appeal to \autoref{thm:generalities:classical.completeness.for.regular.categories}.
\end{proof}

\begin{lem}
\label{lem:hotoposes:internal.Kan.fibrations.and.Ex.for.internal.simplicial.objects}
Let $\mathcal{S}$ be a regular category. The functor $\Ex : \mbfs \mathcal{S} \to \mbfs \mathcal{S}$ preserves internal Kan fibrations. In particular, it preserves internal Kan complexes.
\end{lem}
\begin{proof} \exerproof
By \autoref{thm:generalities:classical.completeness.for.regular.categories}, it suffices to prove the claim in the case where $\mathcal{S} = \cat{\Set}$, which is well known.\footnote{See Lemma 3.4 in \citep{Kan:1957a} or Lemma 4.5 in \citep[\Chap III]{GJ}.}
\end{proof}

\begin{lem}
\label{lem:hotoposes:Ex.is.internally.weakly.homotopy.equivalent}
For any internal Kan complex $X$ in a regular category $\mathcal{S}$, the canonical embedding $i_X : X \to \Ex{X}$ is a Dugger--Isaksen weak equivalence.
\end{lem}
\begin{proof} \exerproof
By \autoref{thm:generalities:classical.completeness.for.regular.categories}, it suffices to prove the claim in the case where $\mathcal{S} = \cat{\Set}$, which is well known.\footnote{See Lemma 3.7 in \citep{Kan:1957a} or Theorem 4.6 in \citep[\Chap III]{GJ}.}
\end{proof}

\makenumpar
Let $\mathcal{S}$ be a locally small category with limits for finite diagrams and colimits for $\omega$-sequences. For each simplicial object $X$ in $\mathcal{S}$, we define $\Ex[\infty]{X}$ to be the colimit of the diagram below:
\[
\begin{tikzcd}
X \rar{i_X} &
\Ex{X} \rar{i_{\Ex{X}}} &
\Ex[2]{X} \rar{i_{\Ex[2]{X}}} &
\Ex[3]{X} \rar &
\cdots
\end{tikzcd}
\]
The above defines a functor $\Ex[\infty] : \mbfs \mathcal{S} \to \mbfs \mathcal{S}$ and a natural transformation $i^\infty : \id_{\mbfs \mathcal{S}} \hoto \Ex[\infty]$.

\begin{lem}
\label{lem:hotoposes:internal.Kan.fibrations.and.sequential.colimits}
Let $\mathcal{S}$ be a regular category with colimits for $\omega$-sequences. If $\indlim : \Func{\omega}{\mathcal{S}} \to \mathcal{S}$ preserves finite limits, then the following classes of morphisms are closed under colimits for $\omega$-sequences in $\mbfs \mathcal{S}$:
\begin{itemize}
\item The class of internal Kan fibrations of simplicial objects in $\mathcal{S}$.

\item The class of internal trivial Kan fibrations of simplicial objects in $\mathcal{S}$.

\item The class of Dugger--Isaksen weak equivalences of internal Kan complexes in $\mathcal{S}$.
\end{itemize}
\end{lem}
\begin{proof}
Colimits commute with colimits, so $\indlim : \Func{\omega}{\mathcal{S}} \to \mathcal{S}$ always preserves regular epimorphisms. Thus, the hypothesis implies $\indlim : \Func{\omega}{\mathcal{S}} \to \mathcal{S}$ is a regular functor. On the other hand, the functor $\Func{\omega}{\mathcal{S}} \to \Func{\ob \omega}{\mathcal{S}}$ induced by restriction along the inclusion $\ob \omega \embedinto \omega$ is a conservative regular functor, so the internal Kan fibrations (\resp internal trivial Kan fibrations, Dugger--Isaksen weak equivalences) in $\Func{\omega}{\mathcal{S}}$ are just the componentwise ones. Thus, the indicated classes of morphisms in $\mbfs \mathcal{S}$ are closed under colimits for $\omega$-sequences.
\end{proof}

\begin{thm}
\label{thm:hotoposes:Ex-infty.functor.for.simplicial.objects}
Let $\mathcal{S}$ be a regular category with colimits for $\omega$-sequences. If $\indlim : \Func{\omega}{\mathcal{S}} \to \mathcal{S}$ preserves finite limits, then:
\begin{enumerate}[(i)]
\item The functor $\Ex[\infty] : \mbfs \mathcal{S} \to \mbfs \mathcal{S}$ preserves finite limits and internal Kan fibrations.

\item For any simplicial object $X$ in $\mathcal{S}$, the simplicial object $\Ex[\infty]{X}$ is an internal Kan complex.

\item For any internal Kan complex $X$ in $\mathcal{S}$, the morphism $i^\infty_X : X \to \Ex[\infty]{X}$ is Dugger--Isaksen weak equivalence.
\end{enumerate}
\end{thm}
\begin{proof}
(i). It is clear that $\Ex : \mbfs \mathcal{S} \to \mbfs \mathcal{S}$ preserves finite limits, and by hypothesis, $\indlim : \Func{\omega}{\mathcal{S}} \to \mathcal{S}$ also preserves finite limits, so the same is true for $\Ex[\infty] : \mbfs \mathcal{S} \to \mbfs \mathcal{S}$. The preservation of internal Kan fibrations is a consequence of lemmas \ref{lem:hotoposes:internal.Kan.fibrations.and.Ex.for.internal.simplicial.objects} and \ref{lem:hotoposes:internal.Kan.fibrations.and.sequential.colimits}.

\bigskip\noindent
(ii). If $\indlim : \Func{\omega}{\mathcal{S}} \to \mathcal{S}$ preserves finite limits, then for any finite simplicial set $Z$, the functor $\kwlim{Z}{\blank} : \mbfs \mathcal{S} \to \mathcal{S}$ preserves colimits for $\omega$-sequences. In particular, we have a commutative diagram of the form below,
\[
\begin{tikzcd}
\indlim_{m : \omega} \kwlim{\Delta^n}{\Ex[m+2]{X}} \dar \rar{\cong} &
\kwlim{\Delta^n}{\Ex[\infty]{X}} \dar \\
\indlim_{m : \omega} \kwlim{\Lambda^n_k}{\Ex[m+2]{X}} \rar[swap]{\cong} &
\kwlim{\Lambda^n_k}{\Ex[\infty]{X}}
\end{tikzcd}
\]
where the horizontal arrows are the canonical comparisons and the vertical arrows are induced by the horn inclusion $\Lambda^n_k \embedinto \Delta^n$. \Autoref{lem:hotoposes:horn.fillers.and.Ex.for.internal.simplicial.objects} gives us the following pullback square in $\mathcal{S}$,
\[
\begin{tikzcd}
\kwlim{\Lambda^n_k}{\Ex[m+1]{X}} \times_{\kwlim{\Lambda^n_k}{\Ex[m+2]{X}}} \kwlim{\Delta^n}{\Ex[m+2]{X}} \dar \rar &
\kwlim{\Delta^n}{\Ex[m+2]{X}} \dar \\
\kwlim{\Lambda^n_k}{\Ex[m+1]{X}} \rar[swap]{\kwlim{\Lambda^n_k}{i_{\Ex[m+1]{X}}}} &
\kwlim{\Lambda^n_k}{\Ex[m+2]{X}}
\end{tikzcd}
\]
where $\kwlim{\Lambda^n_k}{\Ex[m+1]{X}} \times_{\kwlim{\Lambda^n_k}{\Ex[m+2]{X}}} \kwlim{\Delta^n}{\Ex[m+2]{X}}  \to \kwlim{\Lambda^n_k}{\Ex[m+1]{X}}$ is a regular epimorphism in $\mathcal{S}$. It is easy to see that
\[
\indlim_{m : \omega} {\kwlim{\Lambda^n_k}{i_{\Ex[m+1]{X}}}} : \indlim_{m : \omega} \kwlim{\Lambda^n_k}{\Ex[m+1]{X}} \to \indlim_{m : \omega} \kwlim{\Lambda^n_k}{\Ex[m+2]{X}}
\]
is an isomorphism in $\mathcal{S}$, so $\kwlim{\Delta^n}{\Ex[\infty]{X}} \to \kwlim{\Lambda^n_k}{\Ex[\infty]{X}}$ is indeed a regular epimorphism in $\mathcal{S}$, as required.

\bigskip\noindent
(iii). \Autoref{thm:hotoposes:category.of.internally.fibrant.objects} and \autoref{lem:hotoposes:Ex.is.internally.weakly.homotopy.equivalent}
imply that the composite morphism
\[
\begin{tikzcd}
X \rar{i_X} &
\Ex{X} \rar &
\cdots \rar &
\Ex[m]{X} \rar{i_{\Ex[m]{X}}} &
\Ex[m+1]{X}
\end{tikzcd}
\] 
is a Dugger--Isaksen weak equivalence, and since $i^\infty_X  : X \to \Ex[\infty]{X}$ is a colimit for the $\omega$-sequence of these composites, we may apply \autoref{lem:hotoposes:internal.Kan.fibrations.and.sequential.colimits} to deduce that it is also a Dugger--Isaksen weak equivalence.
\end{proof}

Following a suggestion of \citet[\Chap II, Introduction]{Quillen:1967}, we can define `weak homotopy equivalence' using Kan's $\Ex[\infty]$ functor to replace simplicial objects with internal Kan complexes. For reasons that will become clear in the next section, we will only make this definition when the base category is sufficiently well behaved.

\begin{dfn}
A \strong{$\sigma$-pretopos} is a category $\mathcal{C}$ with these properties:
\begin{itemize}
\item $\mathcal{C}$ is an effective regular category.

\item $\mathcal{C}$ has coproducts for countable families of objects, and these are moreover disjoint and pullback-stable.
\end{itemize}
\end{dfn}

\begin{remark}
By Theorem 5.15 in \citep{Shulman:2012c}, a $\sigma$-pretopos in the sense above is automatically $\sigma$-coherent; hence, this definition agrees with the one given by \citet[Part A, \Sect 1.4]{Johnstone:2002a}. 
\end{remark}

\begin{remark}
\label{rem:hotoposes:sigma.pretoposes.and.countable.colimits}
A $\sigma$-pretopos has coequalisers for all parallel pairs: see Lemma 1.4.19 in \citep[Part A]{Johnstone:2002a}. Thus, $\sigma$-pretoposes have colimits for all countable diagrams. Similarly, by Lemma 2.5.7 in \opcit, a $\sigma$-coherent functor preserves colimits for countable diagrams. Thus, by embedding a (small) $\sigma$-pretopos in the topos of sheaves for the $\sigma$-coherent topology, we find that colimits for $\omega$-sequences in a $\sigma$-pretopos commute with finite limits, as required in the hypothesis of \autoref{thm:hotoposes:Ex-infty.functor.for.simplicial.objects}. (See also Proposition 5.1.8 in \citep[Part D]{Johnstone:2002b}.)
\end{remark}

\begin{dfn}
A \strong{Kan weak equivalence} of simplicial objects in a $\sigma$-pretopos $\mathcal{S}$ is a morphism $f : X \to Y$ in $\mbfs \mathcal{S}$ such that the induced morphism $\Ex[\infty]{f} : \Ex[\infty]{X} \to \Ex[\infty]{Y}$ is a Dugger--Isaksen weak equivalence.
\end{dfn}

\begin{remark}
\label{rem:hotoposes:new.and.old.internal.whe.of.Kan.complexes}
Recalling the 2-out-of-3 property of Dugger--Isaksen weak equivalences,  \autoref{thm:hotoposes:Ex-infty.functor.for.simplicial.objects} implies that the Kan weak equivalences of internal Kan complexes are precisely the Dugger--Isaksen weak equivalences.
\end{remark}

To compare Kan weak equivalences of simplicial sets with weak homotopy equivalences in the usual sense, we require the following result:

\begin{lem}
\label{lem:ssets:Ex-infty.is.weakly.homotopy.equivalent}
For any simplicial set $X$, the canonical embedding $i^\infty_X : X \to \Ex[\infty]{X}$ is a weak homotopy equivalence.
\end{lem}
\begin{proof} \openproof
See Lemma 6.5 in \citep{Kan:1957a} or Theorem 4.6 in \citep[\Chap III]{GJ}.
\end{proof}

\begin{cor}
\label{cor:hotoposes:internal.whe.of.simplicial.sets}
A morphism of simplicial sets is a Kan weak equivalence if and only if it is a weak homotopy equivalence. \qed
\end{cor}

Unfortunately, as the construction of $\Ex[\infty]$ involves infinite colimits, we cannot apply the classical completeness theorem for regular logic to transfer results from $\cat{\Set}$ to more general settings. Instead, we must use more subtle machinery.

\section{Weak homotopy equivalences}

Recall the following results:

\begin{thm}[Raptis--Rosický]
\label{thm:ssets:Raptis-Rosicky}
Let $\cat{\SSet}$ be the category of simplicial sets, let $\Func{\mathbf{2}}{\cat{\SSet}}$ be the category of morphisms, and let $\mathcal{W}$ be the full subcategory of $\Func{\mathbf{2}}{\cat{\SSet}}$ spanned by the weak homotopy equivalences.
\begin{enumerate}[(i)]
\item $\mathcal{W}$ is closed under filtered colimits in $\Func{\mathbf{2}}{\cat{\SSet}}$.

\item The weak homotopy equivalences of finite simplicial sets constitute an essentially small solution set for the inclusion $\mathcal{W} \embedinto \Func{\mathbf{2}}{\cat{\SSet}}$.

\item $\mathcal{W}$ is a finitely accessible category, and the inclusion $\mathcal{W} \embedinto \Func{\mathbf{2}}{\cat{\SSet}}$ preserves finitely presentable objects.
\end{enumerate}
\end{thm}
\begin{proof} \openproof
See Theorem A in \citep{Raptis-Rosicky:2014}.
\end{proof}

\begin{thm}[Makkai--Paré]
\label{thm:logic:theories.of.presheaf.type}
Let $\mathcal{C}$ be a finitely accessible category, let $\Kompakt{\mathcal{C}}$ be the full subcategory of finitely presentable objects in $\mathcal{C}$, and let $\mathcal{B} = \Func{\Kompakt{\mathcal{C}}}{\cat{\Set}}$. Then $\mathcal{C}$ is equivalent to the full subcategory of $\Func{\mathcal{B}}{\cat{\Set}}$ spanned by those functors $\mathcal{B} \to \cat{\Set}$ that preserve limits for finite diagrams and colimits for all diagrams, where the equivalence sends an object $C$ in $\mathcal{C}$ to a functor $\mathcal{B} \to \cat{\Set}$ that sends the representable functor $\Hom[\Kompakt{\mathcal{C}}]{T}{\blank}$ to the set $\Hom[\mathcal{C}]{T}{C}$.
\end{thm}
\begin{proof} \openproof
See Corollary 1.2.5 and Proposition 2.1.8 in \citep{Makkai-Pare:1989}, or Theorem 2.26 in \citep{LPAC}.
\end{proof}

The topos $\mathcal{B}$ constructed in the Makkai--Paré theorem is, loosely speaking, a \strong{classifying topos} for $\mathcal{C}$.\footnote{Strictly speaking, the notion of classifying topos is only defined for categories ``parametrised'' by Grothendieck toposes, and these categories are not always determined by their realisation over $\cat{\Set}$. Nonetheless, for a finitely accessible category $\mathcal{C}$, one can define the category of $\mathcal{C}$-objects in any Grothendieck topos, and \emph{this} is what $\mathcal{B}$ classifies.} For instance, if $\mathcal{C} = \cat{\Set}$, then $\mathcal{B}$ classifies objects; if $\mathcal{C} = \cat{\SSet}$, then $\mathcal{B}$ classifies simplicial objects; and if $\mathcal{C} = \Func{\mathbf{2}}{\cat{\SSet}}$, then $\mathcal{B}$ classifies morphisms between simplicial objects. In each case, $\mathcal{B}$ contains a universal instance of the kind of structure being classified, and this structure can be pulled back along a geometric morphism (with codomain $\mathcal{B}$) to obtain a structure on the domain. 

Now, let $\mathcal{B}_\mathrm{whe}$ be classifying topos for weak homotopy equivalences and consider the universal weak homotopy equivalence $u : A \to B$. This is a morphism between simplicial objects in $\mathcal{B}_\mathrm{whe}$, and by \autoref{thm:logic:theories.of.presheaf.type}, we may identify $\mathcal{B}_\mathrm{whe}$ with a topos of the form $\Func{\mathcal{K}}{\cat{\Set}}$, where $\mathcal{K}$ is a small category. Since limits and colimits in $\Func{\mathcal{K}}{\cat{\Set}}$ are computed componentwise, it follows that components of $u : A \to B$ are weak homotopy equivalences of simplicial sets. We make the following definition:

\begin{dfn}
An \strong{internal weak homotopy equivalence} of simplicial objects in a Grothendieck topos $\mathcal{S}$ is a morphism $f : X \to Y$ in $\mbfs \mathcal{S}$ that is isomorphic to one of the form $c^* u : c^* A \to c^* B$ for some geometric morphism $c : \mathcal{S} \to \mathcal{B}_\mathrm{whe}$.
\end{dfn}

\begin{remark}
If $\mathcal{S} = \cat{\Set}$, then an internal weak homotopy equivalence is just a weak homotopy equivalence in the usual sense, by \autoref{thm:logic:theories.of.presheaf.type}. If $\mathcal{S} = \Func{\op{\mathcal{C}}}{\cat{\Set}}$ for a small category $\mathcal{C}$, then an internal weak homotopy equivalence is the same thing as a componentwise weak homotopy equivalence, because the evaluation functors $\Func{\op{\mathcal{C}}}{\cat{\Set}} \to \cat{\Set}$ preserve limits and colimits and are jointly conservative.
\end{remark}

\begin{lem}
\label{lem:hotoposes:Raptis-Rosicky.weak.equivalences.are.Kan.weak.equivalences}
Every internal weak homotopy equivalence of simplicial objects in a Grothendieck topos is also a Kan weak equivalence.
\end{lem}
\begin{proof}
Since inverse image functors preserve Dugger--Isaksen weak equivalences and commute with $\Ex[\infty]$, they also preserve Kan weak equivalences; thus it suffices to prove the claim for the universal weak homotopy equivalence $u : A \to B$ in the classifying topos $\mathcal{B}_\mathrm{whe}$.

By \autoref{thm:ssets:Raptis-Rosicky} and \autoref{thm:logic:theories.of.presheaf.type}, $\mathcal{B}_\mathrm{whe}$ is a presheaf topos, so the class of internal weak homotopy equivalences of simplicial objects in $\mathcal{B}_\mathrm{whe}$ has the 2-out-of-3 property. In particular, $\Ex[\infty]{u} : \Ex[\infty]{A} \to \Ex[\infty]{B}$ is an internal weak homotopy equivalence of internal Kan complexes, hence also a Dugger--Isaksen weak equivalence. Thus, $u : A \to B$ itself is a Kan weak equivalence, as required.
\end{proof}

\begin{prop}
\label{prop:hotoposes:Ex-infty.is.internally.weakly.homotopy.equivalent}
For any simplicial object $X$ is a Grothendieck topos $\mathcal{S}$, the canonical embedding $i^\infty_X : X \to \Ex[\infty]{X}$ is an internal weak homotopy equivalence (hence, a Kan weak equivalence).
\end{prop}
\begin{proof} \exerproof
Since inverse image functors preserve Kan weak equivalences, it suffices to prove the claim for the universal simplicial object in the classifying topos for simplicial objects. 

Let $\mathcal{K}$ be a small skeleton of the category of finite simplicial sets and let $\mathcal{B}_\mathrm{ss} = \Func{\mathcal{K}}{\cat{\Set}}$. It can be shown that $\mathcal{B}_\mathrm{ss}$ is the classifying topos for simplicial objects. Clearly, $\mathcal{B}_\mathrm{ss}$ is a topos with enough points, so the claim can be reduced to the case where $\mathcal{S} = \cat{\Set}$, which is \autoref{lem:ssets:Ex-infty.is.weakly.homotopy.equivalent}.
\end{proof}

\begin{cor}
\label{cor:hotoposes:Ex-infty.in.sigma.pretoposes}
For any simplicial object $X$ is a $\sigma$-pretopos $\mathcal{S}$, the canonical embedding $i^\infty_X : X \to \Ex[\infty]{X}$ is a Kan weak equivalence.
\end{cor}
\begin{proof}
Since the claim only involves a small family of objects and morphisms in $\mathcal{S}$,  we can replace $\mathcal{S}$ with a small $\sigma$-pretopos; and if $\mathcal{S}$ is a small $\sigma$-pretopos, we can embed it in the topos of sheaves for the $\sigma$-coherent topology, as in  \autoref{rem:hotoposes:sigma.pretoposes.and.countable.colimits}. Thus, it suffices to prove the claim in the case where $\mathcal{S}$ is a Grothendieck topos.
\end{proof}

\begin{prop}
\label{prop:hotoposes:comparison.of.whe.of.internal.Kan.complexes}
\needspace{3.0\baselineskip}
Let $\mathcal{S}$ be a Grothendieck topos and let $f : X \to Y$ be a morphism in $\cat{\Kancx[\mathcal{S}]}$. The following are equivalent:
\begin{enumerate}[(i)]
\item $f : X \to Y$ is a Kan weak equivalence.

\item $f : X \to Y$ is a Dugger--Isaksen weak equivalence.

\item $f : X \to Y$ is an internal weak homotopy equivalence.
\end{enumerate}
\end{prop}
\begin{proof}
(i) \iff (ii). This is \autoref{rem:hotoposes:new.and.old.internal.whe.of.Kan.complexes}.

\bigskip\noindent
(ii) \implies (iii). Consider the classifying topos $\mathcal{B}_\mathrm{DI}$ for Dugger--Isaksen weak equivalences of Kan complexes. It is a subtopos of the classifying topos for morphisms of simplicial objects, and the corresponding Grothendieck topology is generated by singletons. Thus, by Deligne's theorem on coherent toposes,\footnote{See Proposition 9.0 in \citep[Exposé VI]{SGA4b}, Corollary 3 in \citep[\Chap IX, \Sect 11]{MLM}, or Proposition 3.3.13 in \citep[Part D]{Johnstone:2002b}.} $\mathcal{B}_\mathrm{DI}$ is a topos with enough points. We can therefore reduce the claim to the case where $\mathcal{S} = \cat{\Set}$, where it is known that Dugger--Isaksen weak equivalences are weak homotopy equivalences in the classical sense.

\bigskip\noindent
(iii) \implies (i). See \autoref{lem:hotoposes:Raptis-Rosicky.weak.equivalences.are.Kan.weak.equivalences}.
\end{proof}

Next, we will show that the class of internal weak homotopy equivalences coincides with the class of Kan weak equivalences; for this we will need some technical results from topos theory.

\begin{lem}
\label{lem:logic:products.of.presheaf.toposes}
Let $\mathcal{C}$ and $\mathcal{D}$ be small categories. The bicategorical product of $\Func{\op{\mathcal{C}}}{\cat{\Set}}$ and $\Func{\op{\mathcal{D}}}{\cat{\Set}}$ in the 2-category of Grothendieck toposes is canonically equivalent to $\Func{\op{\mathcal{C}} \times \op{\mathcal{D}}}{\cat{\Set}}$.
\end{lem}
\begin{proof} \openproof
See Exercise 14 in \citep[\Chap VII]{MLM} or Corollary 3.2.13 in \citep[Part B]{Johnstone:2002a}.
\end{proof}

\begin{lem}
\label{lem:logic:topologies.of.presheaf.type}
Let $\mathcal{C}$ and $\mathcal{D}$ be small categories. If $F : \mathcal{C} \to \mathcal{D}$ is a fully faithful functor, then:
\begin{enumerate}[(i)]
\item The geometric morphism $\Func{\op{\mathcal{C}}}{\cat{\Set}} \to \Func{\op{\mathcal{D}}}{\cat{\Set}}$ covariantly induced by $F$ is a geometric inclusion.

\item The corresponding Grothendieck topology on $\mathcal{D}$ has the following property: there is a unique minimal covering sieve on each object in $\mathcal{D}$.
\end{enumerate}
\end{lem}
\begin{proof} \openproof
(i). The counit of pointwise right Kan extension along a fully faithful functor is a natural isomorphism: see Corollary 3 in \citep[\Chap X, \Sect 3]{CWM}. In particular, the direct image functor $\Func{\op{\mathcal{C}}}{\cat{\Set}} \to \Func{\op{\mathcal{D}}}{\cat{\Set}}$ induced by $F$ is fully faithful.

\bigskip\noindent
(ii). See the remarks after Definition 2.2.18 in \citep[Part C]{Johnstone:2002b}.
\end{proof}

\begin{thm}[Makkai--Reyes]
\label{thm:logic:Makkai-Reyes.completeness}
If $\tuple{\mathcal{C}, J}$ is a separable site, \ie $\mathcal{C}$ is a countable category with finite limits and $J$ is a Grothendieck topology generated by countably many sieves, then the topos $\cat{\Sh{\mathcal{C}}{J}}$ has enough points.
\end{thm}
\begin{proof} \openproof
See Theorem 6.2.4 in \citep{Makkai-Reyes:1977}.
\end{proof}

\begin{lem}
\label{lem:logic:pullback.topologies}
Let $\tuple{\mathcal{D}, K}$ be a Grothendieck site, let $\mathcal{C}$ be a category, let $F : \mathcal{C} \to \mathcal{D}$ be a functor, and let $J$ be the smallest Grothendieck topology on $\mathcal{C}$ such that $F : \mathcal{C} \to \mathcal{D}$ reflects covers.\footnote{--- \ie if $V$ is a $K$-covering sieve on $F C$, then its preimage is a $J$-covering sieve on $C$.} Then there is a bicategorical pullback diagram of the form below in the 2-category of Grothendieck toposes,
\[
\begin{tikzcd}
\cat{\Sh{\mathcal{C}}{J}} \dar[hookrightarrow] \rar &
\cat{\Sh{\mathcal{D}}{K}} \dar[hookrightarrow] \\
\Func{\op{\mathcal{C}}}{\cat{\Set}} \rar &
\Func{\op{\mathcal{D}}}{\cat{\Set}}
\end{tikzcd}
\]
where $\Func{\op{\mathcal{C}}}{\cat{\Set}} \to \Func{\op{\mathcal{D}}}{\cat{\Set}}$ is the geometric morphism covariantly induced by $F : \mathcal{C} \to \mathcal{D}$.
\end{lem}
\begin{proof} \openproof
See the proof of Lemma 2.3.19 in \citep[Part C]{Johnstone:2002b}.
\end{proof}
%
%
%

\begin{lem}
\label{lem:logic:fibrations.of.sites}
Let $\tuple{\mathcal{C}, J}$ and $\tuple{\mathcal{D}, K}$ be Grothendieck sites. Given an adjunction of the form below,
\[
F \dashv G : \mathcal{D} \to \mathcal{C}
\] 
the left adjoint $F : \mathcal{C} \to \mathcal{D}$ reflects covers if and only if the right adjoint $G : \mathcal{D} \to \mathcal{C}$ preserves covers (\ie $G : \tuple{\mathcal{D}, K} \to \tuple{\mathcal{C}, J}$ is a morphism of sites).
\end{lem}
\begin{proof} \openproof
See Lemma 2.5.1 in \citep[Part C]{Johnstone:2002a}.
\end{proof}

\begin{cor}
\label{cor:logic:separability.of.pullbacks}
Let $\tuple{\mathcal{D}, K}$ be a Grothendieck site, let $\mathcal{C}$ be a category, let $F : \mathcal{C} \to \mathcal{D}$ be a functor, and let $J$ be the smallest Grothendieck topology on $\mathcal{C}$ such that $F : \mathcal{C} \to \mathcal{D}$ reflects covers. Assume the following hypotheses:
\begin{itemize}
\item $\mathcal{C}$ is a countable category with finite limits.

\item $\tuple{\mathcal{D}, K}$ is a separable site.

\item $F : \mathcal{C} \to \mathcal{D}$ has a right adjoint, say $G : \mathcal{D} \to \mathcal{C}$.
\end{itemize}
Then $\tuple{\mathcal{C}, J}$ is also a separable site.
\end{cor}
\begin{proof}
By hypothesis, $K$ is generated by a countable collection of sieves in $\mathcal{D}$, say $K_0$. Let $J_0$ be the collection of sieves in $\mathcal{C}$ that are generated by the image of some sieve that are in $K_0$ and let $J_1$ be the Grothendieck topology generated by $J_0$. By \autoref{lem:logic:fibrations.of.sites}, $J$ is the smallest Grothendieck topology on $\mathcal{C}$ such that $G : \mathcal{D} \to \mathcal{C}$ preserves covers; thus, every sieve that is in $J_0$ is $J$-covering, and therefore $J_1 \subseteq J$. We will show (by structural induction) that $G : \tuple{\mathcal{D}, K} \to \tuple{\mathcal{C}, J_1}$ is a morphism of sites; it then follows that $J_1 = J$. 

Suppose $V$ is a $K$-covering sieve on $D$ and $g : D' \to D$ is a morphism in $\mathcal{D}$. Then the pullback sieve $g^* V$ is a $K$-covering sieve on $D'$. Moreover, $\mathcal{D}$ has pullbacks and $G : \mathcal{D} \to \mathcal{C}$ preserves pullbacks, so the sieve on $G D'$ generated by the image of $g^* V$ is also the pullback along $G g : G D' \to G D$ of the sieve generated by the image of $V$. Thus, the image of $g^* V$ generates a $J_1$-covering sieve if the image of $V$ generates a $J_1$-covering sieve.

It is clear that $G : \mathcal{D} \to \mathcal{C}$ preserves multi-composition of sieves and is inclusion-preserving on sieves. Since $K$ is the closure of $K_0$ under pullback,  multi-composition, and upward inclusion, the above proves that the image of each $K$-covering sieve generates a $J_1$-covering sieve, so we are done.
\end{proof}

\begin{prop}
\label{prop:hotoposes:2-of-3.property.of.internal.whe}
The class of internal weak homotopy equivalences of simplicial objects in a Grothendieck topos $\mathcal{S}$ has the 2-out-of-3 property.
\end{prop}
\begin{proof} \exerproof
The proposition consists of three claims, and it suffices to verify each claim for the universal example in the appropriate classifying topos. Each of these classifying toposes admits a site of the form $\tuple{\mathcal{C}, J}$ where $\mathcal{C}$ is the \emph{opposite} of the category of commutative triangles in the category of finite simplicial sets and $J$ is a suitable Grothendieck topology.

For example, if we wished to show that the class of internal weak homotopy equivalences is closed under composition, $J$ would be the smallest Grothendieck topology on $\mathcal{C}$ such that the functor $F : \mathcal{C} \to \mathcal{D}$ reflects covers, where $\mathcal{D}$ is the category of pairs of morphisms of finite simplicial sets, $F : \mathcal{C} \to \mathcal{D}$ sends a commutative triangle to the underlying composable pair, and $\mathcal{D}$ is equipped with the Grothendieck topology $K$ such that $\cat{\Sh{\mathcal{D}}{K}}$ is the classifying topos for pairs of weak homotopy equivalences of simplicial objects; then \autoref{lem:logic:pullback.topologies} ensures that $\cat{\Sh{\mathcal{C}}{J}}$ is the classifying topos for composable pairs of weak homotopy equivalences of simplicial objects. Moreover, $F : \mathcal{C} \to \mathcal{D}$ has a right adjoint, so by lemmas~\ref{lem:logic:products.of.presheaf.toposes} and~\ref{lem:logic:topologies.of.presheaf.type} and \autoref{cor:logic:separability.of.pullbacks}, $\tuple{\mathcal{C}, J}$ is (equivalent to) a separable site. Thus, by \autoref{thm:logic:Makkai-Reyes.completeness}, the classifying toposes in question have enough points, and therefore the claims can be reduced to the case where $\mathcal{S} = \cat{\Set}$.
\end{proof}

At last, we arrive at the main theorem of this section:

\begin{thm}
\label{thm:hotoposes:comparison.of.whe.in.Grothendieck.toposes}
A morphism of simplicial objects in a Grothendieck topos $\mathcal{S}$ is an internal weak homotopy equivalence if and only if it is a Kan weak equivalence.
\end{thm}
\begin{proof}
The `only if' direction was \autoref{lem:hotoposes:Raptis-Rosicky.weak.equivalences.are.Kan.weak.equivalences}; we will now prove the `if' direction.

Suppose $f : X \to Y$ is a Kan weak equivalence of simplicial objects in $\mathcal{S}$. That means $\Ex[\infty]{f} : \Ex[\infty]{X} \to \Ex[\infty]{Y}$ is a Dugger--Isaksen weak equivalence of internal Kan complexes. But the following diagram in $\mbfs \mathcal{S}$ commutes,
\[
\begin{tikzcd}
X \dar[swap]{f} \rar{i^\infty_X} &
\Ex[\infty]{X} \dar{\Ex[\infty]{f}} \\
Y \rar[swap]{i^\infty_Y} &
\Ex[\infty]{Y}
\end{tikzcd}
\]
where the horizontal arrows are internal weak homotopy equivalences by \autoref{prop:hotoposes:Ex-infty.is.internally.weakly.homotopy.equivalent}; hence, applying propositions \ref{prop:hotoposes:comparison.of.whe.of.internal.Kan.complexes} and \ref{prop:hotoposes:2-of-3.property.of.internal.whe}, we deduce that $f : X \to Y$ is indeed an internal weak homotopy equivalence.
\end{proof}

\section{Further properties}

First, we will prove that an internal Kan fibration of simplicial objects in a $\sigma$-pretopos is an internal trivial Kan fibration if and only if it is a Kan weak equivalence. For this, we need one more technical lemma from topos theory:

\begin{lem}
\label{lem:logic:separability.of.intersections}
Let $\mathcal{C}$ be a small category and let $J_0$ and $J_1$ be Grothendieck topologies on $\mathcal{C}$. If $J_2$ is the Grothendieck topology on $\mathcal{C}$ generated by the union of $J_0$ and $J_1$, then we have the following bicategorical pullback diagram in the 2-category of Grothendieck toposes:
\[
\begin{tikzcd}
\cat{\Sh{\mathcal{C}}{J_2}} \dar[hookrightarrow] \rar[hookrightarrow] &
\cat{\Sh{\mathcal{C}}{J_1}} \dar[hookrightarrow] \\
\cat{\Sh{\mathcal{C}}{J_0}} \rar[hookrightarrow] &
\Func{\op{\mathcal{C}}}{\cat{\Set}}
\end{tikzcd}
\]
In particular, if $\tuple{\mathcal{C}, J_0}$ and $\tuple{\mathcal{C}, J_1}$ are separable sites, then so is $\tuple{\mathcal{C}, J_2}$.
\end{lem}
\begin{proof}
Let $J$ be any Grothendieck topology on $\mathcal{C}$ and let $L : \mathcal{C} \to \cat{\Sh{\mathcal{C}}{J}}$ be the functor that sends an object $C$ to the sheaf associated with the representable presheaf $\Hom[\mathcal{C}]{\blank}{C}$. For any Grothendieck topos $\mathcal{E}$, the functor that sends a geometric morphism $x : \mathcal{E} \to \cat{\Sh{\mathcal{C}}{J}}$ to the functor $x^* L : \mathcal{C} \to \mathcal{E}$ is fully faithful and essentially surjective onto the full subcategory of torsors $\mathcal{C} \to \mathcal{E}$ that send $J$-covering sieves in $\mathcal{C}$ to jointly epimorphic families in $\mathcal{E}$, \ie a site morphism $\tuple{\mathcal{C}, J} \to \tuple{\mathcal{E}, K}$ where $K$ is the canonical topology on $\mathcal{E}$.\footnote{See Corollary 2 in \citep[\Chap VII, \Sect 10]{MLM} or Corollary 2.3.9 in \citep[Part C]{Johnstone:2002b}.} But it is clear that a torsor $\mathcal{C} \to \mathcal{E}$ sends $J_0$- and $J_1$-covering sieves in $\mathcal{C}$ to jointly epimorphic families in $\mathcal{E}$ if and only if it is a site morphism $\tuple{\mathcal{C}, J_2} \to \tuple{\mathcal{E}, K}$, so we are done.\footnote{Note that we are using the fact that the intersection of two full \emph{replete} subcategories has the universal property of a bicategorical pullback: see Proposition 5.1.1 in \citep{Makkai-Pare:1989}.}
\end{proof}

\begin{prop}
\label{prop:hotoposes:internal.trivial.Kan.fibrations.in.Grothendieck.toposes}
Let $f : X \to Y$ be a morphism of simplicial objects in a Grothendieck topos $\mathcal{S}$. The following are equivalent:
\begin{enumerate}[(i)]
\item $f : X \to Y$ is an internal trivial Kan fibration.

\item $f : X \to Y$ is an internal Kan fibration and a Kan weak equivalence.

\item $f : X \to Y$ is an internal Kan fibration and an internal weak homotopy equivalence.

\end{enumerate}
\end{prop}
\begin{proof} \exerproof
It suffices to verify each claim for the universal example in the appropriate classifying topos.

\bigskip\noindent
(i) \implies (ii). Let $\mathcal{B}_\mathrm{t.fib}$ be the classifying for trivial Kan fibrations. $\mathcal{B}_\mathrm{t.fib}$ is a subtopos of the classifying topos for morphisms of simplicial objects and corresponds to a Grothendieck topology generated by singletons, so by Deligne's theorem on coherent toposes, $\mathcal{B}_\mathrm{t.fib}$ has enough points. Thus, we may reduce the claim to the case where $\mathcal{S} = \cat{\Set}$, which is well known.\footnote{See Theorem 11.2 in \citep[\Chap I]{GJ}.}

\bigskip\noindent
(ii) \implies (iii). Use \autoref{thm:hotoposes:comparison.of.whe.in.Grothendieck.toposes}.

\bigskip\noindent
(iii) \implies (i). Let $\mathcal{B}_\mathrm{fib}$ be the classifying topos for Kan fibrations. $\mathcal{B}_\mathrm{fib}$ is a subtopos of the classifying topos for morphisms of simplicial objects and corresponds to a Grothendieck topology generated by singletons, and by \autoref{lem:logic:topologies.of.presheaf.type}, $\mathcal{B}_\mathrm{whe}$ corresponds to a Grothendieck topology generated by countably many sieves. Applying \autoref{lem:logic:separability.of.intersections} and \autoref{thm:logic:Makkai-Reyes.completeness}, we deduce that $\mathcal{B}_\mathrm{fib} \cap \mathcal{B}_\mathrm{whe}$ is a topos with enough points. Once again, we may reduce the claim to the case where $\mathcal{S} = \cat{\Set}$, which is well known.
\end{proof}

\begin{cor}
\label{cor:hotoposes:internal.trivial.Kan.fibrations.in.sigma.pretoposes}
\needspace{3.0\baselineskip}
Let $f : X \to Y$ be a morphism of simplicial objects in a $\sigma$-pretopos. The following are equivalent:
\begin{enumerate}[(i)]
\item $f : X \to Y$ is an internal trivial Kan fibration.

\item $f : X \to Y$ is an internal Kan fibration and a Kan weak equivalence.
\end{enumerate}
\end{cor}
\begin{proof}
Use the embedding theorem for $\sigma$-pretoposes (\autoref{rem:hotoposes:sigma.pretoposes.and.countable.colimits}) to reduce to the case where $\mathcal{S}$ is a Grothendieck topos.
\end{proof}

\needspace{3.0\baselineskip}
Next, we will show that Kan weak equivalences are well behaved with respect to pullbacks and pushouts. 

\begin{prop}
\label{prop:hotoposes:right.properness.of.internal.whe}
Let $\mathcal{S}$ be a $\sigma$-pretopos. Consider a pullback diagram in $\mbfs \mathcal{S}$:
\[
\begin{tikzcd}
X' \dar \rar{f} &
X \dar{p} \\
Y' \rar[swap]{g} &
Y
\end{tikzcd}
\]
If $p : X \to Y$ is an internal Kan fibration and $g : Y' \to Y$ is a Kan weak equivalence, then $f : X' \to X$ is also a Kan weak equivalence.
\end{prop}
\begin{proof}
By \autoref{thm:hotoposes:Ex-infty.functor.for.simplicial.objects}, $\Ex[\infty] : \mbfs \mathcal{S} \to \cat{\Kancx[\mathcal{S}]}$ preserves pullbacks and internal Kan fibrations, and by \autoref{cor:hotoposes:Ex-infty.in.sigma.pretoposes}, it preserves and \emph{reflects} Kan weak equivalences. Thus, we may assume that $X, Y, X', Y'$ are all in $\cat{\Kancx[\mathcal{S}]}$; but \autoref{thm:hotoposes:category.of.internally.fibrant.objects} says $\cat{\Kancx[\mathcal{S}]}$ is a category of fibrant objects, and the claim is known to be true in any category of fibrant objects.\footnote{See Lemma 2 in \citep[\Sect 4]{Brown:1973} or Lemma 8.5 in \citep[\Chap II]{GJ}.}
\end{proof}

\begin{prop}
\label{prop:hotoposes:left.properness.of.internal.whe}
Let $\mathcal{S}$ be a Grothendieck topos. Consider a pushout diagram in $\mbfs \mathcal{S}$:
\[
\begin{tikzcd}
X \dar[swap]{i} \rar{f} &
X' \dar \\
Y \rar[swap]{g} &
Y'
\end{tikzcd}
\]
If $i : X \to Y$ is a monomorphism and $f : X \to X'$ is an internal weak homotopy equivalence, then $g : Y \to Y'$ is also an internal weak homotopy equivalence.
\end{prop}
\begin{proof} \exerproof
Let $\mathcal{C}$ be the \emph{opposite} of the category of spans in the category of finite simplicial sets. It is not hard to see that $\Func{\op{\mathcal{C}}}{\cat{\Set}}$ is the classifying topos for spans of simplicial objects. The classifying topos for monomorphisms of simplicial objects is also a presheaf topos (namely, the topos presheaves on the \emph{opposite} of the category of monomorphisms of finite simplicial sets), so by applying \autoref{lem:logic:topologies.of.presheaf.type}, \autoref{cor:logic:separability.of.pullbacks}, and~\autoref{lem:logic:separability.of.intersections}, we deduce that there is a Grothendieck topology $J$ such that $\tuple{\mathcal{C}, J}$ is a separable site and $\mathcal{B} = \cat{\Sh{\mathcal{C}}{J}}$ classifies spans where one leg is a monomorphism and the other leg is a weak homotopy equivalence. It suffices to prove the claim for the universal such span in $\mathcal{B}$; but by \autoref{thm:logic:Makkai-Reyes.completeness}, $\mathcal{B}$ has enough points, so the claim is reduced to the case where $\mathcal{S} = \cat{\Set}$, which is well known.\footnote{See Corollary 8.6 in \citep[\Chap II]{GJ}.}
\end{proof}

\begin{cor}
\label{cor:hotoposes:left.properness.of.internal.whe.in.sigma.pretoposes}
Let $\mathcal{S}$ be a $\sigma$-pretopos. Consider a pushout diagram in $\mbfs \mathcal{S}$:
\[
\begin{tikzcd}
X \dar[swap]{i} \rar{f} &
X' \dar{i'} \\
Y \rar[swap]{g} &
Y'
\end{tikzcd}
\]
If $i : X \to Y$ is a monomorphism and $f : X \to X'$ is an internal weak homotopy equivalence, then $g : Y \to Y'$ is also an internal weak homotopy equivalence.
\end{cor}
\begin{proof}
Use the embedding theorem for $\sigma$-pretoposes (\autoref{rem:hotoposes:sigma.pretoposes.and.countable.colimits}) to reduce to the case where $\mathcal{S}$ is a Grothendieck topos.
\end{proof}

\section{Local homotopy theory}

\makenumpar
Throughout this section, $\tuple{\mathcal{C}, J}$ is a small Grothendieck site. There is a well-known adjunction of the form below,
\[
j^* \dashv j_* : \cat{\Sh{\mathcal{C}}{J}} \to \Func{\op{\mathcal{C}}}{\cat{\Set}}
\]
where the right adjoint $j_*$ is the inclusion and the left adjoint $j^*$ preserves finite limits.\footnote{See Theorem 1 in \citep[\Chap III, \Sect 5]{MLM} or Proposition 2.2.6 in \citep[Part C]{Johnstone:2002b}.} For brevity, we will write $\cat{\SPsh{\mathcal{C}}}$ for the category of simplicial presheaves on $\mathcal{C}$.

\begin{dfn}
\ \noprelistbreak
\begin{itemize}
\item A \strong{$J$-local epimorphism} (\resp \strong{$J$-local isomorphism}) of presheaves on $\mathcal{C}$ is a morphism $f : X \to Y$ in $\Func{\op{\mathcal{C}}}{\cat{\Set}}$ such that $j^* f : j^* X \to j^* Y$ is an epimorphism (\resp isomorphism) in $\cat{\Sh{\mathcal{C}}{J}}$.

\item A \strong{$J$-local Kan fibration} (\resp \strong{$J$-local trivial Kan fibration}, \strong{$J$-local weak homotopy equivalence}) of simplicial presheaves on $\mathcal{C}$ is a morphism $f : X \to Y$ in $\cat{\SPsh{\mathcal{C}}}$ such that $j^* : j^* X \to j^* Y$ is an internal Kan fibration (\resp internal trivial Kan fibration, internal weak homotopy equivalence) of simplicial objects in $\cat{\Sh{\mathcal{C}}{J}}$.

\item A \strong{$J$-locally fibrant simplicial presheaf} is an object $X$ in $\cat{\SPsh{\mathcal{C}}}$ such that $j^* X$ is an internal Kan complex in $\cat{\Sh{\mathcal{C}}{J}}$.
\end{itemize}
\end{dfn}

The above definitions agree with other definitions found in the literature. To prove this, we require a well-known lemma:

\begin{lem}
\label{lem:logic:local.epimorphisms.and.covering.sieves}
\needspace{3\baselineskip}
Let $f : X \to Y$ be a morphism in $\Func{\op{\mathcal{C}}}{\cat{\Set}}$. The following are equivalent:
\begin{enumerate}[(i)]
\item $f : X \to Y$ is a $J$-local epimorphism.

\item For each object $C$ in $\mathcal{C}$ and each element $y$ of $Y \argp{C}$, there exists a $J$-covering sieve $U$ on $C$ such that, for each morphism $u : U' \to U$ that is in $U$, there is an element $x_u$ of $X \argp{C'}$ such that $f_{C'} \argp{x_u} = u^* y$.

\item For each object $C$ in $\mathcal{C}$ and each element $y$ of $Y \argp{C}$, if $U$ is the sieve of morphisms $u : C' \to C$ in $\mathcal{C}$ for which there is an element $x_u$ of $X \argp{C'}$ such that $f_{C'} \argp{x_u} = u^* y$, then $U$ is a $J$-covering sieve.
\end{enumerate}
\end{lem}
\begin{proof} \openproof
See Corollary 5 in \citep[\Chap III, \Sect 8]{MLM}.
\end{proof}

\begin{thm}
\ \noprelistbreak
\begin{enumerate}[(i)]
\item A morphism of simplicial presheaves is a $J$-local Kan fibration if and only if it is a local fibration in the sense of \citet{Jardine:1987}.

\item A morphism of simplicial presheaves is a $J$-local weak homotopy equivalence if and only if it is a topological weak equivalence in the sense of \citet{Jardine:1987} or a local weak equivalence in the sense of \citet{Dugger-Isaksen:2004b}.

\item A morphism of simplicial presheaves is a $J$-local trivial Kan fibration if and only if it is a local acyclic fibration in the sense of \citet{Dugger-Isaksen:2004b}.
\end{enumerate}
\end{thm}
\begin{proof}
(i). Since $j^* : \Func{\op{\mathcal{C}}}{\cat{\Set}} \to \cat{\Sh{\mathcal{C}}{J}}$ preserves finite limits, a morphism $f : X \to Y$ in $\cat{\SPsh{\mathcal{C}}}$ is a $J$-local Kan fibration if and only if the comparison morphism
\[
\kwlim{\Delta^n}{X} \to \kwlim{\Lambda^n_k}{X} \times_{\kwlim{\Lambda^n_k}{Y}} \kwlim{\Delta^n}{Y}
\]
induced by the commutative diagram in $\Func{\op{\mathcal{C}}}{\cat{\Set}}$
\[
\begin{tikzcd}
\kwlim{\Delta^n}{X} \dar[swap]{\kwlim{i}{X}} \rar{\kwlim{\Delta^n}{p}} &
\kwlim{\Delta^n}{Y} \dar{\kwlim{i}{Y}} \\
\kwlim{\Lambda^n_k}{X} \rar[swap]{\kwlim{\Lambda^n_k}{p}} &
\kwlim{\Lambda^n_k}{Y}
\end{tikzcd}
\]
is a $J$-local epimorphism for every horn inclusion $i : \Lambda^n_k \to \Delta^n$. Thus, recalling lemmas \ref{lem:generalities:lifting.properties.and.pullback.products} and \ref{lem:logic:local.epimorphisms.and.covering.sieves}, our $J$-local Kan fibrations are the same as what Jardine calls `local fibrations'.

\bigskip\noindent
(ii). The homotopical uniqueness of fibrant replacement in $\cat{\SSet}$ implies a morphism of simplicial presheaves is a topological weak equivalence in the sense of Jardine if and only if it is a local weak equivalence in the sense of Dugger and Isaksen. Theorem 6.15 in \citep{Dugger-Isaksen:2004b} says that a morphism $f : X \to Y$ in $\cat{\SPsh{\mathcal{C}}}$ is a local weak equivalence if and only if $j^* \Ex[\infty]{f} : j^* \Ex[\infty]{X} \to j^* \Ex[\infty]{Y}$ is a Dugger--Isaksen weak equivalence (in the sense of \autoref{dfn:hotoposes:DI.whe}) of internal Kan complexes in $\cat{\Sh{\mathcal{C}}{J}}$. Since $j^*$ preserves finite limits and all colimits, it commutes with $\Ex[\infty]$; we thus conclude that local weak equivalences in the sense of Dugger and Isaksen are the same as $J$-local weak homotopy equivalences.

\bigskip\noindent
(iii). Local acyclic fibrations are (by definition) the same thing as local fibrations that are local weak equivalences. In view of (i) and (ii), it suffices to verify that a $J$-local trivial Kan fibration is precisely a $J$-local Kan fibration that is a $J$-local weak homotopy equivalence; but this is an immediate consequence of \autoref{prop:hotoposes:internal.trivial.Kan.fibrations.in.Grothendieck.toposes}.
\end{proof}

\begin{remark}
It is worth emphasising that our $J$-local trivial Kan fibrations have a local lifting property \emph{by definition}. We have therefore obtained a new proof of Proposition 7.2 in \citep{Dugger-Isaksen:2004b}.
\end{remark}

\begin{remark}
Since the adjunction counit $j^* j_* \hoto \id_{\cat{\Sh{\mathcal{C}}{J}}}$ is a natural isomorphism, the triangle identity implies the adjunction unit $\id_{\Func{\op{\mathcal{C}}}{\cat{\Set}}} \hoto j_* j^*$ is a natural $J$-local isomorphism. In particular, for any simplicial presheaf $X$, the canonical morphism $X \to j_* j^* X$ is a $J$-local weak homotopy equivalence. Compare Lemma 2.6 in \citep{Jardine:1987}. 
\end{remark}

\begin{remark}
The class of $J$-local weak homotopy equivalences is closed under pullback along $J$-local fibrations, by \autoref{prop:hotoposes:right.properness.of.internal.whe}, and it is also closed under pushout along monomorphisms, by \autoref{prop:hotoposes:left.properness.of.internal.whe}. The latter fact can also be proved by using the fact that the cofibrations in Jardine's model structure are precisely the monomorphisms of simplicial presheaves.
\end{remark}

\needspace{3.0\baselineskip}
We conclude with a new proof of an old result:

\begin{thm}
The full subcategory $\mathcal{W}_J$ of $\Func{\mathbf{2}}{\cat{\SPsh{\mathcal{C}}}}$ spanned by the $J$-local weak homotopy equivalences is closed under filtered colimits and is an accessible category.
\end{thm}
\begin{proof}
First, consider the full subcategory $\mathcal{W}$ of $\Func{\mathbf{2}}{\cat{\SSh{\mathcal{C}}{J}}}$ spanned by the internal weak homotopy equivalences. By definition, it is equivalent to the category of geometric morphisms $\cat{\Sh{\mathcal{C}}{J}} \to \mathcal{B}_\mathrm{whe}$, and such categories are known to be accessible.\footnote{See Remark 2.3.12 in \citep[Part D]{Johnstone:2002b}.} If we then identify geometric morphisms $\cat{\Sh{\mathcal{C}}{J}} \to \mathcal{B}_\mathrm{whe}$ with torsors $\op{\mathcal{K}} \to \cat{\Sh{\mathcal{C}}{J}}$, where $\mathcal{K}$ is the category of weak equivalences of finite simplicial sets, it is straightforward to see that $\mathcal{W}$ is closed under filtered colimits as a subcategory of $\Func{\mathbf{2}}{\cat{\SSh{\mathcal{C}}{J}}}$.

Since $j^* : \Func{\op{\mathcal{C}}}{\cat{\Set}} \to \cat{\Sh{\mathcal{C}}{J}}$ is a left adjoint, the induced functor $j^* : \Func{\mathbf{2}}{\cat{\SPsh{\mathcal{C}}}} \to \Func{\mathbf{2}}{\cat{\SSh{\mathcal{C}}{J}}}$ is accessible. It is clear that $\mathcal{W}_J$ is the preimage of $\mathcal{W}$ under $j^*$, so it is closed under filtered colimits. Moreover, $\mathcal{W}$ is a full replete subcategory of $\Func{\mathbf{2}}{\cat{\SSh{\mathcal{C}}{J}}}$, so we may apply the pseudopullback theorem\footnote{See Theorem 5.1.6 in \citep{Makkai-Pare:1989} or Exercise 2.n in \citep{LPAC}.} to deduce that $\mathcal{W}_J$ is also accessible.
\end{proof}

\ifdraftdoc

\else
  \printbibliography
\fi

\end{document}